\documentclass[12pt]{amsart}

\usepackage{amsmath}
\usepackage{amssymb}
\usepackage{amsthm}
\usepackage{mathtools}

\usepackage{caption}
\usepackage{floatrow}%

\usepackage[margin=2cm]{geometry}

\newcommand{\questionref}[1]{\hyperref[#1]{Ques\-tion~\ref*{#1}}}
\newcommand{\theoremref}[1]{\hyperref[#1]{The\-o\-rem~\ref*{#1}}}
\newcommand{\claimref}[1]{\hyperref[#1]{Claim~\ref*{#1}}}
\newcommand{\situationref}[1]{\hyperref[#1]{Situation~\ref*{#1}}}
\newcommand{\lemmaref}[1]{\hyperref[#1]{Lem\-ma~\ref*{#1}}}
\newcommand{\remarkref}[1]{\hyperref[#1]{Re\-mark~\ref*{#1}}}
\newcommand{\notationref}[1]{\hyperref[#1]{No\-ta\-tion~\ref*{#1}}}
\newcommand{\sectionref}[1]{\hyperref[#1]{Section~\ref*{#1}}}
\newcommand{\subsectionref}[1]{\hyperref[#1]{Sub\-sec\-tion~\ref*{#1}}}
\newcommand{\definitionref}[1]{\hyperref[#1]{Def\-i\-ni\-btion~\ref*{#1}}}
\newcommand{\propositionref}[1]{\hyperref[#1]{Prop\-o\-si\-tion~\ref*{#1}}}
\newcommand{\conjectureref}[1]{\hyperref[#1]{Con\-jec\-ture~\ref*{#1}}}
\newcommand{\corollaryref}[1]{\hyperref[#1]{Cor\-ol\-lar\-y~\ref*{#1}}}
\newcommand{\exerciseref}[1]{\hyperref[#1]{Ex\-er\-cise~\ref*{#1}}}
\newcommand{\figureref}[1]{\hyperref[#1]{Fig\-ure~\ref*{#1}}}

\usepackage[mathscr]{euscript}

\usepackage[all]{xy}

\xyoption{2cell}
\UseAllTwocells

\usepackage{multicol}

\usepackage{xr-hyper}

\usepackage{hyperref}
\usepackage[usenames,dvipsnames]{xcolor}
\hypersetup{colorlinks=true,citecolor=NavyBlue,linkcolor=BrickRed,urlcolor=Orange}

\theoremstyle{plain}
\newtheorem{theorem}{Theorem}[section]
\newtheorem{proposition}[theorem]{Proposition}
\newtheorem{lemma}[theorem]{Lemma}
\newtheorem{corollary}[theorem]{Corollary}

\theoremstyle{definition}
\newtheorem{definition}[theorem]{Definition}

\theoremstyle{remark}
\newtheorem{remark}[theorem]{Remark}

\newtheorem{notation}[theorem]{Notation}

\newtheorem{question}[theorem]{Question}

\usepackage{tikz}
\usepackage{tikz-cd}
\usetikzlibrary{matrix,arrows}
\usetikzlibrary{decorations.pathmorphing}

\def\lcm{\mathop{\mathrm{lcm}}\nolimits}

\DeclareMathOperator{\Pic}{Pic}

\DeclareMathOperator{\im}{Im}

\DeclareMathOperator{\Proj}{Proj}

\AtBeginDocument{%
   \def\MR#1{}
}

\usepackage{amscd}

\def\Pic{\operatorname{Pic}}

\newcommand{\bP}{\mathbb{P}}

\newcommand{\bZ}{\mathbb{Z}}
\newcommand{\bQ}{\mathbb{Q}}

\newcommand{\bC}{\mathbb{C}}

\newcommand{\Stab}{\mathrm{Stab}}

\newcommand{\calD}{\mathcal{D}} 
\newcommand{\calH}{\mathcal{H}} 

\newcommand{\calP}{\mathcal{P}}

\newcommand{\calO}{\mathcal{O}}

\newcommand{\calM}{\mathcal{M}}

\begin{document}

\title[Compactifications of the moduli space of plane quartics and two lines]{Compactifications of the moduli space of plane quartics and two lines}
\author{Patricio Gallardo}
\address{Department of Mathematics, Washington University in St. Louis}
\email{pgallardocandela@wustl.edu}
\author{Jesus Martinez-Garcia}
\address{Department of Mathematical Sciences, University of Bath.}
\email{J.Martinez.Garcia@bath.ac.uk}
\author{Zheng Zhang}
\address{Department of Mathematics, Texas A\&M University.}
\email{zzhang@math.tamu.edu}
\date{April 7, 2018}

\bibliographystyle{amsalpha}
\begin{abstract}
	We study the moduli space of triples $(C, L_1, L_2)$ consisting of quartic curves $C$ and lines $L_1$ and $L_2$. Specifically, we construct and compactify the moduli space in two ways: via geometric invariant theory (GIT) and by using the period map of certain lattice polarized $K3$ surfaces. The GIT construction depends on two parameters $t_1$ and $t_2$ which correspond to the choice of a linearization. For $t_1=t_2=1$ we describe the GIT moduli explicitly and relate it to the construction via $K3$ surfaces. 
\end{abstract}

\maketitle

\section{Introduction}
The construction of compact moduli spaces with geometric meanings is an important problem in algebraic geometry. In this article, we discuss the case of the moduli of $K3$ surfaces of degree 2 obtained as minimal resolutions of double covers of $\bP^2$ branched at a quartic $C$ and two lines $L_1$, $L_2$, for which we give two constructions, one via Geometric Invariant Theory (GIT) for the plane curves $(C, L_1, L_2)$ depending on a choice of two parameters for each of the lines, and one via the period map of $K3$ surfaces. For a particular choice of parameters, we show that the constructions agree. Similar examples include \cite{shah1980complete}, \cite{ls_cubic}, \cite{looijenga_cubic}, \cite{laza_n16, laza_cubic} and \cite{act_cubicsurface, act_cubic3fold}. 
Our interest on this example arose after the first two authors considered studying the variations of GIT quotients for a cubic surface and a hyperplane section \cite{gallardo2016moduli}. The moduli of del Pezzo surfaces of degree $2$ with two anti-canonical sections seems to be closely related to the moduli of $K3$ surfaces considered in this article, since del Pezzo surfaces of degree 2 with canonical singularities can be obtained as double-covers of $\mathbb P^2$ branched at a (possibly singular) quartic curve. Also, a generic global Torelli for certain double covers of these $K3$ surfaces (namely, minimal resolutions of bi-double covers of $\bP^2$ along a quartic and four lines, cf. \cite[\S 5.4.2]{garbagnati_doublecoverk3}) can be derived using the results in this article and the methods in \cite{pearlsteinzhang}. 
 
Following the general theory of variations of GIT quotients developed by Dolgachev and Hu \cite{dolgachev-hu-vgit} and independently by Thaddeus \cite{thaddeus-vgit}, we construct GIT compactifications $\overline{\calM}(t_1,t_2)$ for the moduli space of triples $(C, L_1, L_2)$ consisting of a smooth plane quartic curve $C$ and two labeled lines $L_1$, $L_2$ in \sectionref{sec:VGIT}. These compactifications depend on parameters $t_1,t_2$ which are the ratio polarizations of the parameter spaces of quartic and linear homogeneous forms representing $C$ and $L_1,L_2$. We generalize the study in \cite{gallardo2016variations} of GIT quotients of pairs $(X,H)$ formed by a hypersurface $X$ of degree $d$ in $\mathbb P^{n+1}$ and a hyperplane $H$ to tuples $(X, H_1, \ldots, H_k)$ with several hyperplanes $H_i$, considering the relation between the moduli spaces of tuples with labeled and unlabeled hyperplanes. We then apply the setting to the case at hand, namely plane quartic curves and two lines. One sees in \lemmaref{intVGIT} that the space where the set of stable points is not empty can be precisely described. Furthermore, given a particular tuple, we can bound the set of parameters for which it is semistable (cf. \lemmaref{lemma:2gen}).

Next we focus on the case when $t_1=t_2=1$. The moduli space $\overline{\calM}(1,1)$ 
can also be constructed via Hodge theory (cf. \sectionref{sec:Hodge}). The idea is to consider the $K3$ surface $S_{(C, L_1, L_2)}$ obtained by taking the desingularization of the double cover $\bar{S}_{(C,L_1,L_2)}$ of $\bP^2$ branched along the sextic curve $C+L_1+L_2$. Note that generically $\bar{S}_{(C,L_1,L_2)}$ admits nine ordinary double points (coming from the intersection points $C \cap L_1$, $C \cap L_2$ and $L_1 \cap L_2$). It follows that the $K3$ surface $S_{(C,L_1,L_2)}$ contains nine $(-2)$-curves which form a certain configuration. Call the saturated sublattice generated by these curves $M \subset \Pic(S_{(C, L_1, L_2)})$. Then the $K3$ surface $S_{(C,L_1,L_2)}$ is naturally $M$-polarized in the sense of Dolgachev \cite{dolgachev_MSK3}. Let $\calM_0 \subset \overline{\calM}(1,1)$ be the locus where the sextic curves $C+L_1+L_2$ have at worst simple singularities (also known as \emph{ADE singularities} or \emph{Du Val singularities}). By associating to the triples $(C,L_1,L_2)$ the periods of the $M$-polarized $K3$ surfaces $S_{(C,L_1,L_2)}$ one obtains a period map $\calP$ from $\calM_0$ to a certain period domain $\calD/\Gamma$. We shall prove that $\calP$ is an isomorphism. 

\begin{theorem}[\theoremref{P isomorphic}]
Consider the triples $(C, L_1, L_2)$ consisting of quartic curves $C$ and lines $L_1$, $L_2$ such that $C + L_1 + L_2$ has at worst simple singularities. Let $S_{(C,L_1,L_2)}$ be the $K3$ surface obtained by taking the minimal resolution of the double plane branched along $C + L_1 +L_2$. The map sending $(C, L_1, L_2)$ to the periods of $S_{(C,L_1,L_2)}$ extends to an isomorphism $\calP: \calM_0 \rightarrow \calD/\Gamma$.
\end{theorem}

The approach is analogue to the one used by Laza \cite{laza_n16}. Roughly speaking, we first consider the generic case where $C$ is smooth and $C+L_1+L_2$ has simple normal crossings. Then we compute the (generic) Picard lattice $M$ and the transcendental lattice $T = M_{\Lambda_{K3}}^{\perp}$ (see \propositionref{determine M T}), determine the period domain $\calD$ and choose a suitable arithmetic group $\Gamma$ (cf. \sectionref{sec:MK3}, N.B. $\Gamma$ is not the standard arithmetic group $O^*(T)$ used in \cite{dolgachev_MSK3} but an extension of $O^*(T)$). Finally we extend the construction to the non-generic case (using the methods and some results of \cite{laza_n16}) and apply the global Torelli theorem and the surjectivity of the period map for $K3$ surfaces to prove the theorem (cf. \sectionref{normalized embedding} and \sectionref{sec:surjectivity}). 

Note that the period domain $\calD$ is a type IV Hermitian symmetric domain. The arithmetic quotients of $\calD$ admit canonical compactifications called Baily-Borel compactifications. To compare the GIT compactification and the Baily-Borel compactification we consider a slightly different moduli space $\overline{\calM}^*$ (constructed by taking a quotient of the GIT quotient $\overline{\calM}(1,1)$) parameterizing triples $(C,L,L')$ consisting of quartic curves $C$ and unlabeled lines $L$, $L'$. In a similar manner, we construct a period map $\calP'$ and prove that $\calP'$ is an isomorphism between the locus $\calM^*_0 \subset \overline{\calM}^*$ where $C+L+L'$ has at worst simple singularities and a certain locally symmetric domain $\calD/\Gamma'$ (cf. \sectionref{unlabeled P'}). Moreover, we show in \corollaryref{cor:gitcll'} that $\overline{\calM}^* \setminus \mathcal M^*_0$  is the union of three points $\overline{III}(1)$, $\overline{III}(2a)$, $\overline{III}(2b)$ and five rational curves $\overline{II}(1)$, $\overline{II}(2a1)$, $\overline{II}(2a2)$, $\overline{II}(2b)$, $\overline{II}(3)$ whose incidence structure is describe in \figureref{fig:incidence}. The quasi-projective variety $\calM^*_0 \subset \overline{\calM}^*$ has codimension higher than $1$ and hence the period map $\calP'$ extends to the GIT compactification $\overline{\calM}^*$. Note also that $\calP'$ preserves the natural polarizations (the polarization of $\calM^*_0$ is induced by the polarization of the moduli of plane sextics and the polarization of $\calD/\Gamma'$ comes from the polarization of moduli of degree $2$ $K3$ surfaces). A proof similar to \cite[Thm. 7.6]{looijenga} shows that the extension of $\calP'$ induces an isomorphism between the GIT quotient $\overline{\calM}^*$ and the Baily-Borel compactification $(\calD/\Gamma')^*$ (see \sectionref{sec:GITBB}). Some computations and remarks on the Baily-Borel boundary components are also included in the paper (cf. \sectionref{sec:BB}). 
\begin{theorem}
[\theoremref{P'gitbb}]
The period map $\calP': \calM_0^* \rightarrow \calD/\Gamma'$ extends to an isomorphism of projective varieties $\overline{\calP'}: \overline{\calM}^* \stackrel{\cong}{\rightarrow} (\calD/\Gamma')^*$ where $(\calD/\Gamma')^*$ denotes the Baily-Borel compactification of $\calD/\Gamma'$.  
\end{theorem}
\begin{center}
\begin{figure}[h!]
\caption{ 
Incidence relations among the boundary components of the compactification of $\mathcal M^*_0$ in  
$\overline{\calM}^*$. We denote $A \to B$ when the boundary component $B$ is contained in the closure of the boundary component $A$.
}
\label{fig:incidence}
$$
\xymatrix{
\overline{II}(2a2)  \ar[r] 
& \overline{III}(2b)  
& \overline{II}(1)   \ar[l] \ar[rd] & 
\\
\overline{II}(2b) \ar[ru] \ar[rd] &  && 
\overline{III}(1)
\\
\overline{II}(2a1) \ar[r] & 
\overline{III}(2a)  
&\overline{II}(3)\ar[ur]  \ar[l]  &
}
$$
\end{figure}
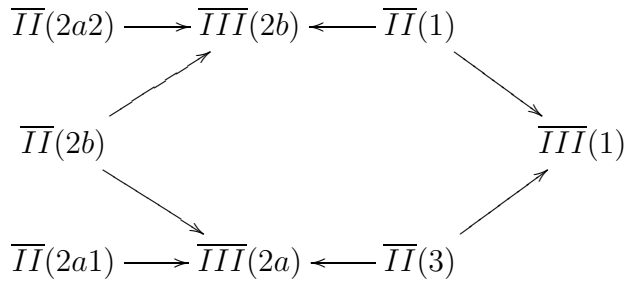
\end{center}

We conclude by the following remarks. The moduli space of quartic triples $(C,L_1,L_2)$ is closely related to the moduli space of degree $5$ pairs (cf. \cite[Def. 2.1]{laza_n16}) consisting of a quintic curve and a line (i.e. given a triple $(C, L_1, L_2)$ that we consider, compare it with the pairs $(C+L_1, L_2)$ and $(C+L_2, L_1)$). Motivated by studying deformations of $N_{16}$ singularities, Laza \cite{laza_n16} has constructed the moduli space of degree $5$ pairs using both the GIT and Hodge theoretic approaches. His work is an important motivation for us and the prototype of what we do here. Also, the study of singularities and incidences lines on quartic curves is a classical topic (see for example the work of Edge \cite{Edge1, Edge2}) and a classifying space for such pairs may be related to our GIT compactification.


\subsection*{Acknowledgments}
We thank Radu Laza and Gregory Pearlstein for useful discussions.  P. Gallardo is supported by the NSF grant DMS-1344994 of the RTG in Algebra, Algebraic Geometry, and Number Theory, at the University of Georgia. J. Martinez-Garcia acknowledges support from the Max Planck Institute for Mathematics in Bonn and would like to thank the Simons Foundation for its support of his research under the Simons Collaboration on Special Holonomy in Geometry, Analysis and Physics (grant \#488631, Johannes Nordstr\"om). This work was partially carried out at the Institute for Computational and Experimental Research in Mathematics (ICERM) during a visit by the authors as part of a Collaborate@ICERM project. The authors would like to thank ICERM for making this visit possible.

\section{Variations of GIT quotients}
\label{sec:VGIT}

In \cite{gallardo2016variations} the first two authors introduced a computational framework to construct all GIT quotients of pairs $(X,H)$ formed by a hypersurface $X$ of degree $d$ and a hyperplane $H$ in $\mathbb P^{n+1}$. They drew from the general theory of variations of GIT quotients developed by Dolgachev and Hu \cite{dolgachev-hu-vgit} and independently by Thaddeus \cite{thaddeus-vgit}. The motivation was to construct compact moduli spaces of log pairs $(X,D=X\cap H)$ where $X$ is Fano or Calabi-Yau. In this article we need to extend this setting to the case of tuples $(C,L_1,L_2)$ where $C$ is a plane quartic curve and $L_1,L_2$ are lines. However, extending our work in \cite{gallardo2016variations} to two hyperplanes entails the same difficulties as for an arbitrary number of hyperplanes, while the dimension does not play an important role in the setting. Therefore we will consider the most general setting of a hypersurface in projective space and $k$ hyperplane sections. 

\subsection{Variations of GIT quotients for $n$-dimensional hypersurfaces of degree $d$ together with $k$ (labelled) hyperplanes}
Let $\mathcal R=\mathcal R_{n,d,k}$  be the  parameter scheme of tuples $(F_d,l_1,\ldots,l_k)$, where $F_d$ is a polynomial of degree $d$ and $l_1,\ldots, l_k$ are linear forms in variables $(x_0,\ldots, x_{n+1})$, modulo scalar multiplication.
We have that
\begin{align*}
\mathcal R_{n,d, k}
& =\mathbb P(H^0(\mathbb P^{n+1}, \mathcal O_{\mathbb P^{n+1}}(d)))\times \mathbb P(H^0(\mathbb P^{n+1}, \mathcal O_{\mathbb P^{n+1}}(1)))\times \cdots \times \mathbb P(H^0(\mathbb P^{n+1}, \mathcal O_{\mathbb P^{n+1}}(1)))
\\
&
\cong \mathbb P^N\times (\mathbb P^{n+1})^k,
\end{align*}
where $N= {\binom{n+1+d}{d}}-1$ and natural projections $\pi_0\colon \mathcal R_{n,d,k}\rightarrow \mathbb P^N$, $\mathbb \pi_i\colon \mathcal R_{n,d,k}\rightarrow \mathbb P^{n+1}$ for $i=1,\ldots, k$. The natural action of $G=\mathrm{SL}_{n+2}$ in $\mathbb P^{n+1}$ extends to each of the factors in $\mathcal R_{n,d,k}$ and therefore to $\mathcal R_{n,d,k}$ itself. The set of $G$-linearizable line bundles $\Pic^G(\mathcal R)$ is isomorphic to $\mathbb{Z}^{n+1}$. Then a line bundle $\mathcal L \in\Pic^{G}(\mathcal R),$ is ample if and only if $a>0$, $b_i>0$ for $i=1,\ldots, k$, where 
$$\mathcal L=\mathcal{O}(a,b_1,\ldots, b_k)\coloneqq\pi_0^*(\mathcal O_{\mathbb P^N}(a))\bigotimes_{i=1}^k\pi_i^*(\mathcal O_{\mathbb P^{n+1}}(b_i))\in\Pic^{G}(\mathcal R).$$
The latter is a trivial generalization of \cite[Lemma 2.1]{gallardo2016variations}. Hence, for $\mathcal L\cong\mathcal O(a,b_1,\ldots, b_k)$,  the  GIT quotient is defined as:
$$\overline{\mathcal M}(t_1,\ldots,t_k)=\overline{\mathcal M}(\vec{t}~)=\overline{\mathcal M}(\vec t~)_{n,d,k}=\Proj\left(\bigoplus_{m\geqslant 0} H^0(\mathcal R,\mathcal L^{\otimes m})^{G}\right),$$
where $t_i=\frac{b_i}{a}$. Next, we explain why it is enough to consider the vector $\vec t=(t_1,\dots, t_k)$ instead of $(a;,b_1,\ldots, b_k)$. Let us introduce some notation.

Given a maximal torus $T\cong \mathbb C^{n+2}\subset G$, we can choose projective coordinates $(x_0,\ldots, x_{n+1})$ such that $T$ is diagonal in $G$. Hence, any one-parameter subgroup $\lambda:\mathbb C^*\rightarrow T$ is a diagonal matrix with diagonal entries $s^{r_i}$ where $r_i\in \mathbb Z$ for all $i$ and $\sum_{i=0}^{n+1} r_i=0$. We say that $\lambda$ is \emph{normalized} if $r_0\geqslant \cdots \geqslant r_{n+1}$ and $\lambda$ is not trivial. Any homogeneous polynomial $g$ of degree $d$ can be written as $g=\sum_{I}g_Ix^I$, where $x^I=x^{d_0}_0\cdots x_{n+1}^{d_{n+1}}$, $I=(d_0,\ldots,d_{n+1})\in \mathbb Z^{n+2}$, $\sum_{i=0}^{n+1}=d$ and $g_I\in \mathbb C$. The support of $g$ is $\mathrm{Supp}(g)=\{x^I\ \colon \ g_I\neq 0\}$. We have a natural pairing $\langle x_0^{d_0}\cdot \cdots \cdot x_{n+1}^{d_{n+1}},\lambda\rangle:=\sum_{i=0}^{n+1}d_ir_i$, which we use to introduce the \emph{Hilbert-Mumford function} for homogeneous polynomials:
$$\mu(g,\lambda)\coloneqq\min\{\langle I,\lambda\rangle \ | \ x^I\in \mathrm{Supp}(g)\}.$$
Define
\begin{align*}
\mu_{\vec t}~((f,l_1,\ldots,l_k),\lambda) &\coloneqq\mu(f,\lambda)+\sum_{i=1}^kt_i\mu(l_i,\lambda),
\end{align*}
which is piecewise linear on $\lambda$ for fixed $(f,l_1,\ldots,l_k)$. Since the Hilbert-Mumford function is functorial\cite[Definition 2.2, cf. p. 49]{Mumford-GIT}, we can generalize \cite[Lemma 2.2]{gallardo2016variations} to show that a tuple $(f,l_1,\ldots,l_k)$ is (semi-)stable with respect to a polarisation $\mathcal L=\mathcal O(a,b_1,\ldots, b_k)$ if and only if
$$\mu^{\mathcal L}((f,l_1,\ldots,l_k),\lambda)=a\mu(f,\lambda)+\sum_{i=1}^kb_i\mu(l_i,\lambda)=a\mu_{\vec t}~((f,l_1,\ldots,l_k),\lambda)$$
is negative (respectively, non-positive) for any normalized non-trivial one-parameter subgroup $\lambda$ of any maximal torus $T$ of $G$. Hence the stability of a tuple is independent of the scaling of $\mathcal L$ and as such, we may define:
\begin{definition}
Let $\vec t\in (\mathbb Q_{\geqslant 0})^k$. 
The tuple $(f,l_1,\ldots,l_k)$ is \emph{$\vec t$-stable} (respectively \emph{$\vec t$-semistable}) if $\mu_{\vec t}~(f,l_1,\cdots,l_k,\lambda)<0$ (respectively $\mu_{\vec t}~(f,l_1,\cdots,l_k,\lambda)\leqslant 0$) for all non-trivial normalized one-parameter subgroups $\lambda$ of $G$. A tuple $(f,l_1,\ldots,l_k)$ is \emph{$\vec t$-unstable} if it is not $\vec t$-semistable. A tuple $(f,l_1,\ldots,l_k)$ is \emph{strictly $\vec t$-semistable} if it is $\vec t$-semistable but not $\vec t$-stable. 
\end{definition}
Notice that the stability of a tuple $(f,l_1,\ldots,l_k)$ is completely determined by the support of $f$ and $l_1,\ldots,l_k$. Moreover, notice that the $\vec t$-stability of a tuple is invariant under the action of $G$. Hence, we may say that a tuple $(X,H_1,\ldots, H_k)$ formed by a hypersurface $X\subset \mathbb  P^{n+1}$ and hyperplanes $H_i\subset \mathbb P^{n+1}$ is \emph{$\vec t$-stable} (respectively, \emph{$\vec t$-semistable}) if some (and hence any) tuple of homogeneous polynomials $(f,l_1,\ldots,l_k)$ defining $(X,H_1,\ldots,H_k)$ is \emph{$\vec t$-stable} (respectively, \emph{$\vec t$-semistable}). A tuple $(X,H_1,\ldots, H_k)$ is \emph{$\vec t$-unstable} if it is not $\vec t$-semistable.

In \cite{gallardo2016variations}, for fixed torus $T$ in $G$, we introduced the \emph{fundamental set} $S_{n,d}$ \emph{of one-parameter subgroups} ---a finite set--- and we showed that if $k=1$ it was sufficient to consider the one-parameter subgroups in $S_{n,d}$ for each $T$ to determine the $\vec t$-stability of any $(X,H_1)$. Let us recall the definition ---slightly simplified from the original \cite[Definition 3.1]{gallardo2016variations}--- and extend the result to any $k$.

\begin{definition}
\label{definition:set-S}
The \emph{fundamental set} $S_{n,d}$ \emph{of one-parameter subgroups} 
$\lambda \in T$ consists of all elements $\lambda=\mathrm{Diag}(s^{r_0},\ldots,s^{r_{n+1}})$ where
$$(r_0,\ldots,r_{n+1})=c(\gamma_0,\ldots,\gamma_{n+1})\in \mathbb Z^{n+1}$$
satisfying the following:
\begin{itemize}
\item[(1)] $\gamma_i=\frac{\alpha_i}{\beta_i}\in \mathbb Q$ such that $\gcd(\alpha_i,\beta_i)=1$ for all $i=0,\ldots,n+1$ and $c=\mathrm{lcm}(\beta_0,\ldots,\beta_{n+1})$.
\item[(2)] $1=\gamma_0\geqslant \gamma_1\geqslant \cdots\geqslant \gamma_{n+1}=-1-\sum_{i=1}^n\gamma_i.$
\item[(3)] $(\gamma_0,\ldots,\gamma_{n+1})$ is the unique solution of a consistent linear system given by $n$ equations chosen from the following set:
\begin{align}
&\mathrm{Eq}(n,d)\coloneqq\left\{\sum_{i=0}^{n+1}\delta_i\gamma_i = 0 \ | \ \delta_i\in \mathbb Z_{\geqslant 0}, -d\leqslant \delta_i\leqslant d \text{ for all } i \text{ and } \sum_{i=0}^{n+1} \delta_i=0\right\}.\nonumber
\end{align}
\end{itemize}
\end{definition}
The set $S_{n,d}$ is finite since there are a finite number of monomials of degree $d$ in $n+2$ variables. Observe that $S_{n,d}$ is independent of the value of $k$. The following lemma is a straight forward generalization of \cite[Lemma 3.2]{gallardo2016variations} which we include here for the convenience of the reader:
\begin{lemma}
\label{lemma:finiteness-lemma}
A tuple $(X,H_1,\ldots,H_k)$ given by equations $(f,l_1,\ldots,l_k)$ is not $\vec t$-stable (respectively not $\vec t$-semistable) if and only if there is $g \in G$ satisfying
$$\mu_{\vec t}~(X,H) \coloneqq \max_{\substack{\lambda\in S_{n,d}}}\{\mu_{\vec t}~((g \cdot f, g \cdot l_1,\ldots, g\cdot l_k),\lambda)\} \geqslant 0 \qquad (\text{respectively }>0).$$
Moreover $S_{n,d}\subseteq S_{n,d+1}$. 
\end{lemma}
\begin{proof}
Let $(R^{ns}_{T})_{\vec t}$ be the non-${\vec t}$-stable loci of $\mathcal R$ with respect to a maximal torus $T$, and let 
$(\mathcal R^{ns})_{\vec t}$ be the non $\vec t$-stable loci of $\mathcal R$.

By \cite[p 137.]{dolgachev-lectures-git}), $(\mathcal R^{ns})_{\vec t} = \bigcup_{T_i \subset G} (R_{T_i}^{ns})_{\vec t}$. Let  $(f,l_1,\ldots,l_k)$ be the equations in some coordinate system ---inducing a maximal torus $T\subset G$--- of a non $\vec t$-stable tuple $(X,H_1,\ldots, H_k)$. Then, $\mu_{\vec t}~((f,l_1,\ldots,l_k),\rho)\geqslant 0$ for some $\rho \in T'$ in a maximal torus $T'$ which may be \emph{different} from $T$. All the maximal tori are conjugate to each other in $G$, and by \cite[Exercise 9.2.(i)]{dolgachev-lectures-git}, we have $\mu_{\vec t}~((f,l_1,\ldots,l_k), \rho) = \mu_{\vec t}~(g \cdot (f,l_1,\ldots,l_k), g\rho g^{-1})$ for all $g\in G$. Hence, there is $g_0 \in G$ such that $\lambda:=g_0\rho g_0^{-1}\in T$ is normalized and $(f',l_1',\ldots,l_k')\coloneqq g_0\cdot(f,l_1,\ldots,l_k)$ satisfies $\mu_{\vec t}~((f',l_1',\ldots,l_k'),\lambda)\geqslant 0$.  Normalized one-parameter subgroups in the coordinate system induced by $T$ are the intersection of $\sum r_i=0$ and the convex hull of $r_i-r_{i+1}\geqslant 0$, where $i=0,\ldots, n$. The restriction of the $n+1$ linearly independent inequalities in $n+1$ variables to $\sum r_i=0$ give a closed convex polyhedral subset $\Delta$ of dimension $n+1$ (in fact, a simplex) in the $\mathbb Q$-lattice of characters of $T$ ---isomorphic to the lattice of monomials (in variables $x_0,\ldots, x_{n+1}$) tensored by $\mathbb Q$, which in turn is isomorphic to $\mathbb Q^{n+2}$. 

Given a fixed $(f,l_1,\ldots,l_k)$, the function $\mu_{\vec t}~((f,l_1,\ldots,l_k),-):\mathbb Q^{n+2}\rightarrow \mathbb Q$ is piecewise linear and its critical points ---the points in $\mathbb Q^{n+2}$ where $\mu_{\vec t}~((f,l_1,\ldots,l_k),-)$ fails to be linear--- correspond to those monomials $x^I,x^{I'}\in \mathrm{Supp}(f)$ such that $\langle x^I, \lambda\rangle = \langle{x^{I'}},\lambda\rangle$, or equivalently, the points $\lambda\in \mathbb Q^{n+2}\cap \Delta$ such that $\langle x^{I-I'},\lambda\rangle=0$ for some $x^I,x^{I'}\in \mathrm{Supp}(f)$. These points define a hyperplane in $\mathbb Q^{n+2}$ and the intersection of this hyperplane with $\Delta$ is a simplex $\Delta_{x^I,x^{I'}}$ of dimension $n$. As $\mu_{\vec t}~((f,l_1,\ldots,l_k),-)$ is linear on the complement of $\Delta_{x^I,x^{I'}}$, the minimum of $\mu_{\vec t}~((f,l_1,\ldots,l_k),-)$ is achieved on the boundary, i.e. either on $\partial\Delta$ or on $\Delta_{x^I,x^{I'}}$ (for some $I$, $I'$), all of which are convex polytopes of dimension $n$. By finite induction, we conclude that the minimum of $\mu_{\vec t}~((f,l_1,\ldots,l_k, -)$ is achieved at one of the vertices of $\Delta$ or $\Delta_{x^I,x^{I'}}$, which correspond precisely, up to multiplication by a constant, to the finite set of one-parameter subgroups in $S_{n,d}$. Indeed, observe that if $\lambda=\mathrm{Diag}(s^{r_0},\ldots, s^{r_{n+1}})$ is one such vertex, then $0=\langle x^{I-I'},\lambda\rangle=\sum_{i=0}^{n+1}\delta_i\gamma_i$ for some  $\delta=(\delta_0,\ldots,\delta_{n+1})=I-I'$ where $\sum_{i=0}^{n+1}=0$ and $-d\leqslant \delta_i\leqslant d$. In addition, observe that we can find one such $\delta$ so that $0=\sum_{i=0}^{n+1}\delta_i\gamma_i=\gamma_i-\gamma_{i+1}$, thus giving the equations determining the maximal facets of $\Delta$, i.e. those where $r_i=r_{i+1}$. The lemma follows from the observation that $\mathrm{Eq}(n,d)\subset\mathrm{Eq}(n,d+1)$.
\end{proof}

\begin{definition}
The \emph{space of GIT stability conditions} is
$$\mathrm{Stab}(n,d,k)\coloneqq\{\vec t\in (\mathbb Q_{\geqslant 0})^k \ \colon \ \text{ there is a }\vec t-\mathrm{semistable }  (X,H_1,\ldots, H_k)\}.$$
\end{definition}

The space of GIT stability conditions is bounded, as it can be realized as a hyperplane section of $\mathrm{Amp}^G(\mathcal R)$. Since $\mathcal R$ is a product of vector spaces (and hence a Mori dream space), $\mathrm{Stab}(n,d,k)$ is also a rational polyhedron. It is possible to precisely describe it and we will do this later for $\mathrm{Stab}(1,d,2)$. Moreover, there is a finite number of non-isomorphic GIT compactifications $\overline{\mathcal M}(\vec t~)$ as $\vec t\in \mathrm{Stab}(n,d,k)$ in varies. Therefore we have a natural division of $\mathrm{Stab}(n,d,k)$ into a finite number of disjoint rational polyhedrons of dimension $k$ called \emph{chambers} and the intersection of any two-chambers is a (possibly empty) rational polyhedron of smaller dimension which we will call a \emph{wall} \cite[Theorem 0.2.3]{dolgachev-hu-vgit}. The quotient $\overline{\mathcal M}(\vec t~)$ is constant as $\vec t$ moves in the interior of a face or chamber. It is possible to find these walls explicitly by means of \lemmaref{lemma:finiteness-lemma} (see \cite[Theorem 1.1]{gallardo2016variations}) for given $(n,d,k)$, since all walls of dimension $k-1$ should be a subset of the finite set of equations
\begin{align}
\label{eq:walls}
\left\{\left. \langle x^I,\lambda\rangle+\sum_{j=1}^kt_j\langle x_{i_j}, \lambda\rangle=0\ \right\vert \ x^I \text{ is a monomial of degree }d,\ 0\leqslant i_j\leqslant n+1,\ \lambda\in S_{n,d} \right\}.
\end{align}

Another interesting feature is that the the $\vec t$-stability of tuples $(X,H_1,\ldots, H_k)$ is equivalent of the $t$-stability of reducible GIT hypersurfaces of higher degree. Indeed:
\begin{lemma}\label{lemma:tuple}
Let $\vec t=(t_1,\ldots, t_k)=(\frac{s_1}{s_1'},\cdots, \frac{s_k}{s_k'}) \in \mathbb Q_{\geqslant 0}$ where $\gcd(s_i,s_i')=1$ for all $i=1,\ldots, k$. Let $I=\{i_1,\ldots, i_l\}$, $I'=\{i_1',\ldots, i_{k-l}'\}$ such that $I\sqcup I'=\{1,\ldots,k\}$ and let $s_0=\lcm(s_{i_1}',\cdots s_{i_l}')$. Let $\vec {t'}=(s_0t_{i_1'},\cdots, s_0t_{i_{k-l}'})$. A tuple $(X,H_1,\ldots,H_k)$ is $\vec t$-(semi)stable if and only if the tuple 
$$\left(\left(X^{s_0}+{\frac{s_0s_{i_1}}{s_{i_1}'}}H_{i_1}+\cdots +{\frac{s_0s_{i_l}}{s_{i_l}'}}H_{i_l}\right),H_{i_1'}, \cdots, H_{i_{k-l}'}\right)$$
is $\vec {t'}$-(semi)stable.

In particular, if $t_1,\ldots, t_k$ are natural numbers, $(X,H_1,\ldots,H_k)$ is $\vec t$-(semi)stable if and only if $X+{t_1}H_1+\cdots +{t_k}H_{k}$ (semi)stable in the classical GIT sense.
\end{lemma}
\begin{proof}
Let $\lambda$ be a normalized one-parameter subgroup, $m$ be a positive integer and $g=\sum g_Ix^I$ be a homogeneous polynomial. Let $J$ be such that $\langle x^J,\lambda\rangle = \min\{\langle x^I,\lambda\rangle \ | \ x^I\subset \mathrm{Supp}(g)\}$. Then, since $\lambda$ is normalized, $m\langle x^J,\lambda\rangle = \min\{\langle x^I,\lambda\rangle \ | \ x^I\subset \mathrm{Supp}(g^m)\}$.

Let $(f,l_1,\ldots,l_k)$ be the equations of $(X,H_1,\ldots, H_k)$ under some system of coordinates and let $\lambda$ be a normalized one-parameter subgroup. Using the above observation, the lemma follows from:
\begin{align*}
s_0\mu_{\vec t}~((f,l_1,\ldots, l_k), \lambda)&=s_0\left(\mu(f,\lambda)+\sum_{j=1}^l\frac{s_{i_j}}{s_{i_j}'}\mu(l_{i_j},\lambda)+ \sum_{j=1}^{k-l}\frac{s_{i_j'}}{s_{i_j'}'}\mu(l_{i_j'}, \lambda)\right)\\
&=\mu\left(\left(f^{s_0}\cdot l_{i_1}^{\frac{s_0s_{i_1}}{s_{i_1}'}}\cdot \cdots \cdot l_{i_l}^{\frac{s_0s_{i_l}}{s_{i_l}'}}\right), \lambda\right) + s_0\sum_{j=1}^{k-l}\frac{s_{i_j'}}{s_{i_j'}'}\mu(l_{i_j'}, \lambda)\\
&=\mu_{\vec {t'}}~\left(\left(f^{s_0}\cdot l_{i_1}^{\frac{s_0s_{i_1}}{s_{i_1}'}}\cdot \cdots \cdot l_{i_l}^{\frac{s_0s_{i_l}}{s_{i_l}'}}\right),l_{i_1'}, \cdots, l_{i_{k-l}'}\right)
\end{align*}
\end{proof}

\begin{corollary}\label{cor:forget}
Let $\vec{t} =(t_1,\ldots, t_j, 0,\ldots, 0)$ and $\vec{t'}=(t_1,\ldots, t_j)$, $j\leqslant k$. Then a tuple $(X,H_1,\ldots, H_k)$ 
is $\vec t$-semistable if and only if $(X,H_1, \ldots, H_j)$ is $\vec{t'}$-semistable.
\end{corollary}

\begin{lemma}[{cf. \cite[Corollary 1.2]{gallardo2016moduli}}]
\label{corollary:dimension-moduli}
If the locus of stable points is not empty, and $d \geqslant 3$, then
$$\dim \overline{\mathcal M}({\vec{t}}~)_{n,d,k}={\binom{n+d+1}{d}}-n^2+(k-4)n+k-4.$$
\end{lemma}
\begin{proof}
From \cite[Theorem 2.1]{orlik-solomon-finite-automorphism-hypersurface}, any hypersurface $X=\{f=0\}$ where $f$ is a homogeneous polynomial of degree $d\geqslant 3$ has $\dim(\mathrm{Aut}(f))=0$. Hence, for any tuple $p=(X,H_1,\dots, H_K)$ such that $X$ is smooth and $X\cap H_i$ has simple normal crossings, its stabilizer $G_p$ satisfies
$$0\leqslant \dim(G_p)=\dim(G_X\cap G_{H_1}\cap \cdots \cap G_{H_k})\leqslant \dim( G_X)\leqslant \dim( \mathrm{Aut}(f))=0,$$
where the last equality follows from \cite[Theorem 2.1]{orlik-solomon-finite-automorphism-hypersurface}.
The result follows from \cite[Corollary 6.2]{dolgachev-lectures-git}:
\begin{align*}
\dim\left(\overline{\mathcal M}(\vec t~)_{n,d,t}\right)&=\dim(\mathcal R)-\dim(G)+\min_{p\in \mathcal R}{\dim G_p}=\\
&=\left({\binom{n+1+d}{d}}-1 + k(n+1)\right)-\left((n+2)^2-1\right).
\end{align*}
\end{proof}

Now let us consider the case of the symmetric polarization of $\mathcal R_{n,d,k}$. In order to do so, observe that the group $S_k$ acts on $\mathcal R_{n,d,k}$ by defining the action of $h\in S_k$ as
$$h:(f,l_1,\ldots,l_k)\mapsto (f,l_{h(1)},\ldots, l_{h(k)}).$$
Define $\mathcal R'_{n,d,k}\coloneqq\mathcal R_{n,d,k}/S_k$, which parametrizes classes of tuples $[(f,l_1,\ldots, l_k)]$ up to multiplication by a scalar \emph{and}  permutation of $(l_1, \cdots, l_k)$, i.e. $[(f,l_1,\ldots, l_k)]=[(a_0f,a_1l_{g(1)},\cdots, a_kl_{g(k)})]$ for $g\in S_k$ and $(a_0,\ldots, a_k)\in (\mathbb C^*)^{k+1}$. Hence, we parameterize the same elements as in $\mathcal R_{n,d,k}$ but we \emph{forget} the ordering of the linear forms. In particular $\mathcal R'_{n,d,k}$ parametrizes pairs $[(X,H_1,\ldots, H_k)]$ formed by a hypersurface $X\subset\mathbb P^{n+1}$ of degree $d$ and $k$ unordered hyperplanes. The quotient morphism $\pi\colon \mathcal R_{n,d,k}\rightarrow \mathcal R'_{n,d,k}$ is $G$-equivariant. Let $\mathcal L_*=\mathcal O (a,b_1,\ldots,b_k)\in \mathrm{Pic}(\mathcal R)$ such that $\left(\frac{b_1}{a},\ldots,\frac{b_k}{a}\right)=(1,\ldots,1)$ (i.e. we are considering $\vec t$-stability with respect to $\vec t=\vec t_*\coloneqq(1,\ldots, 1)$). If the condition
\begin{equation}
\vec t_*=(1,\ldots, 1)\in \Stab(n,d,k).
\label{eq:symmetric-polarization}
\end{equation}
holds, then a tuple $(f,l_1,\ldots, l_k)$ is $t_1$-(semi)stable if and only if $(\pi(f),\pi(l_1),\ldots,\pi(l_k))$ is stable with respect to $\overline{\mathcal L_*}\coloneqq(\pi_*(\mathcal L_*))^{\vee \vee}$ by \cite[Theorem 1.1 and p. 48]{Mumford-GIT}. Hence, it is natural to define the GIT quotient
\begin{equation}
\overline{\mathcal M}^*_{n,d,k}=\Proj\bigoplus_{m\geqslant 0} H^0(\mathcal R'_{n,d,k},\overline{\mathcal L_*}^{\otimes m}))^{G},
\label{eq:GIT_unlabeled}
\end{equation}
which is the \emph{GIT quotient of unordered tuples $(X,H_1,\cdots,  H_k)$ with respect to the polarization $\overline{\mathcal L_*}$}. We have a commutative diagram:
$$
\xymatrix{
\mathcal R_{n,d,k}^{ss}\ar@{->}[d]\ar@{->}[rr]^{\pi}&&\mathcal R_{n,d,k}^{'ss}\ar@{->}[d]\\%
\overline{\mathcal M}(\vec t_*~)_{n,d,k}\ar@{->}[rr]^{\pi_*}&& \overline{\mathcal M}^*_{n,d,k}.}
$$
We want to determine all the orbits represented in $\overline{\mathcal M}^*_{n,d,k}$ from the orbits represented in $\overline{\mathcal M}^*_{n,d+l,k-l}$ via $\overline{\mathcal M}(\vec t_*~)_{n,d,k}$.

Choose $1\leqslant j_1< \cdots< j_l\leqslant k$ and define the $G$-equivariant morphism $\phi_{j_1,\ldots,j_l}\colon \mathcal R_{n,d,k}\rightarrow \mathcal R_{(n,d+l,k-l)}$ given by
$$(f,l_1,\ldots, l_k)\mapsto (f\cdot l_{j_1}\cdots l_{j_l}, l_1,\ldots, \hat{l_{j_i}},\ldots, l_{k}).$$
By \lemmaref{lemma:tuple}, we have a commutative diagram:
$$
\xymatrix{
\mathcal R_{n,d+l,k-l}^{ss}\ar@{->}[d]\ar@{<-}[rrr]^{\phi_{j_1,\ldots,j_l}}&&&\mathcal R_{n,d,k}^{ss}\ar@{->}[d]\ar@{->}[rr]^{\pi}&&\mathcal R_{n,d,k}^{'ss}\ar@{->}[d]\\%
\overline{\mathcal M}(\vec t_*~)_{n,d+l,k-l}\ar@{<-}[rrr]^{{\overline\phi_{j_1,\ldots,j_l}}}&&& \overline{\mathcal M}(\vec t_*)_{n,d,k}\ar@{->}[rr]^{\pi_*}&& {\overline{\mathcal M}}^*_{n,d,k}.}
$$

\begin{proposition}
\label{proposition:big-diagram}
Let $1\leqslant j_1< \cdots <j_l\leqslant k$ and suppose that \eqref{eq:symmetric-polarization} holds. An unordered tuple $[(X,H_1,\cdots,H_k)]$ ---where $X$ is a hypersurface of degree $d$ in $\mathbb P^{n+1}$ and $H_1,\ldots, H_k$ are $k$ unordered hyperplanes--- is (semi)stable with respect to $\overline{\mathcal L_*}$ if and only if 
$$
(X+H_{j_1}+\cdots +H_{j_l}, H_1,\ldots, \hat{H}_{j_i},\ldots, H_k)
$$ 
---a pair represented by a tuple in $\mathcal R_{n,d+l,k-l}$--- is $\vec{t_*}$-(semi)stable. Moreover an orbit $O'\in \mathcal R_{n,d,k}^{'ss}$ is closed if and only if and only if $\phi_{j_1,\ldots,j_l}(\pi^{-1}(O'))$ is closed. In addition, an orbit $O \in \mathcal R^{ss}_{n,d,k}$ is closed if and only if $\phi_{j_1,\ldots, j_l}(O)$ is closed.
\end{proposition}
\begin{proof}
Since \eqref{eq:symmetric-polarization} holds, all the spaces in the above diagram are non-empty. As $\pi$ is finite, the pair $[(X,H_1,\cdots,H_k)]$ ---represented by 
the classes of tuples in $\mathcal R'_{n,d,k}$--- is (semi)stable with 
respect to $\overline{\mathcal L_*}$ if and only if every $(X,H_1,
\cdots,H_k)$ in the class $[(X,H_1,\cdots,H_k)]$ is $\vec t_*$-(semi)stable, by \cite[Theorem 1.1 
and p. 48]{Mumford-GIT}. By \lemmaref{lemma:tuple}, $(X,H_1,
\cdots,H_k)$ ---represented by tuples in $\mathcal R_{n,d,k}$--- is $
\vec t_*$-(semi)stable if and only if $(X+H_{j_1}+\cdots +H_{j_l}, 
H_1,\ldots, \hat{H}_{j_i},\ldots, H_k)$ ---represented by tuples in $
\mathcal R'_{n,d+l,k-l}$--- is $\vec t_{*}$-(semi)stable (note that we 
use the notation $\vec t_*$ for vectors with all entries equal $1$, 
whether $\vec t_*$ has $k$ or $k-l$ entries). The last statement 
regarding closed orbits follows from noting that finite morphisms are 
closed, and hence $\phi_{j_1,\ldots, j_l}$ is closed. 
\end{proof}

\subsection{Symmetric GIT quotient of a quartic curve and two lines}
We have seen how to construct GIT quotients $\overline{\mathcal M}(t_1,t_2)\coloneqq \overline{\mathcal M}(t_1,t_2)_{1,d,2}$ for $(t_1,t_2)\in \mathrm{Stab}(t_1,t_2)$. In this section we apply our results to the case of quartic plane curves ($d=4$), but let us first show that our setting satisfies condition \eqref{eq:symmetric-polarization} for arbitrary degree. Hence, for the rest of the article, we assume that $n=1$ and $k=2$.
\begin{lemma}\label{intVGIT}
The space of GIT stability conditions is
 \begin{align} 
\mathrm{Stab}(1,d,2)
:=
\{ (t_1,t_2) \in \mathbb{R}^2 \; | \; 
2t_1-t_2-d \leqslant 0, \; 2t_2-t_1-d \leqslant 0, \; 0 \leqslant t_1, \; 
0 \leqslant t_2
 \}.\label{eq:stab-plane-curves}
 \end{align}
In particular, \eqref{eq:symmetric-polarization} holds.  
\end{lemma}
\begin{proof}
Let $\vec t=(t_1,t_2)$ be a vector and $(C,L_1,L_2)$ be a $\vec t$-semistable tuple. By a choosing an appropriate change of coordinates, we may assume
\begin{align*}
C\coloneqq\{p(x_0,x_1,x_2)=0\},& &L_1 \coloneqq\{x_0=0\}, & &L_2\coloneqq\{l_2=0\}.
\end{align*}
Let $\lambda=\mathrm{Diag}(s^2,s^{-1}, s^{-1})$. Then, as $t_1\geqslant 0, t_2\geqslant 0$, we have
\begin{align*}
0\geqslant \mu_{(t_1,t_2)}((p, x_0,l_2), \lambda) =\mu(p,\lambda)+2t_1+t_2\mu(l_2,\lambda)\geqslant\mu(x_2^d,\lambda)+2t_1+t_2\mu(x_2,\lambda)=-d+2t_1-t_2.
\end{align*}
Similarly, by taking a change of coordinates such that $L_1=\{l_1=0\},\ L_2=\{x_0=0.\}$, we may show that $0\geqslant -d-t_1+2t_2$.

Recall that the space of GIT stability conditions is convex \cite[0.2.1]{dolgachev-hu-vgit}. Hence it is enough to show that all the vertices of the right hand side in \eqref{eq:stab-plane-curves} have a semistable tuple $(C, L_1,L_2)$ (and hence, they belong to $\mathrm{Stab}(1,d,2)$). These vertices correspond to the points $(0,0)$, $(\frac{d}{2}, 0)$, $(0, \frac{d}{2})$ and $(d,d)$. By \corollaryref{cor:forget}, a tuple $(C, L_1,L_2)$ is $(\frac{d}{2},0)$-semistable if and only if $(C, L_1)$ is $\frac{d}{2}$-semistable, but the space of GIT $t$-stability conditions for plane curves and one hyperplane is $[0,\frac{d}{2}]$ \cite[Theorem 1.1]{gallardo2016variations}. A mirrored argument applies for the stability point $(0,\frac{d}{2})$.

Hence, we only need to exhibit a tuple $(C, L_1,L_2)$ which is $(d,d)$-semistable. Let $(C, L_1, L_2)=(\{x_0^d=0\}, \{x_1=0\}, \{x_2=0\})$. By \lemmaref{lemma:tuple}, such a pair is $t$-semistable if and only if the reducible curve $C+dL_1+dL_2$ (defined by the equation $x_0^dx_1^dx_2^d=0$) of degree $3d$ is semistable in the usual GIT sense. The latter follows from the centroid criterion \cite[Lemma 1.5]{gallardo2016variations}.
\end{proof}

There are two natural problems regarding the subdivision of $\mathrm{Stab}(n,d,k)$ into chambers and walls. One of them is to determine the walls and the solution is usually rather heavy computationally and geometrically speaking (see \cite{gallardo2016variations,gallardo2016moduli} for the case $(n,d,k)=(2,3,1)$ and for a partial answer when $k=1$ and $(n,d)$ are arbitrary). Given a tuple $(X, H_1,\cdots, H_k)$ the second problem consists on determining for which chambers and walls this tuple is (semi)stable. This problem may be easier to solve, especially when the answer to the first problem is known. The problem is simpler when $k=1$, as then $\mathrm{Stab}(n,d,1)$ is one-dimensional has a natural order. Nevertheless, we can give a partial answer when $n=1, k=2$ and $d$ is arbitrary.
\begin{definition}
Let $\operatorname{Stab}(C,L_1,L_2) \subset \mathrm{Stab}(1,d,2)$ be the loci such that 
$(t_1,t_2) \in \operatorname{Stab}(C,L_1,L_2) $ if and only $(C,L_1,L_2)$ is 
$t$-semistable.
\end{definition}

\begin{lemma}\label{lemma:2gen}
Suppose that $C$ is a plane curve of degree $d$ whose only singular point $p \in C$ is a linearly semi-quasihomogeneous singularity \cite[Def. 2.21]{laza_n16} with respect to the weights $\vec w=(w_1,w_2)$, $w_1 \geqslant w_2>0$. Suppose further that $C+L_1+L_2$ have simple normal crossings in $\mathbb C\setminus\{p\}$. Let $f$ be the localization of the equation of $f$ at $p$ and $\vec w(f)$ be its weighted degree with respect to $\vec w$.

\begin{enumerate}
\item
Suppose that $p\not\in L_1\cup L_2$. Then
$$
\operatorname{Stab}(C,L_1,L_2)
\subseteq
\bigg\{
(t_1, t_2) \in \mathrm{Stab}(1,d,2) 
\; \bigg| \;
t_1+t_2-\frac{3\vec w(f)}{w_1+w_2}+d \geqslant 0
\bigg\}.
$$
\item
Suppose that $p\not\in L_2$  and $p\in L_1\cap C$. Then
$$
\operatorname{Stab}(C,L_1,L_2)
\subseteq
\bigg\{
(t_1,t_2) \in \mathrm{Stab}(1,d,2) 
\; \bigg | \; 
t_2-
t_1\frac{2w_2-w_1}{w_1+w_2}
-\frac{3\vec w(f_C)}{w_1+w_2}+d
\geqslant 0
\bigg\}.
$$
\end{enumerate}
\end{lemma}
\begin{proof}
We may choose a coordinate system such that $[0:0:1]$ is the singular point of $C$.  We consider the one parameter subgroup 
$\lambda = (2w_1-w_2, 2w_2-w_1,-w_1-w_2)$
which is normalized, as $w_1 \geqslant w_2$.   

The first statement is equivalent to show that 
if  $(t_1, t_2) \in \mathrm{Stab}(1,d,2) $ and
$
t_1+t_2-\frac{3\vec w(f)}{w_1+w_2}+d < 0,
$
then the triple is $(t_1,t_2)$-unstable.

Let  $l_1(x_0,x_1)+x_2$ and $l_2(x_0,x_1)+x_2$ be the equations of the lines $L_1$ and $L_2$, respectively, where $l_1,l_2$ are linear forms. We have that
\begin{align*}
\mu_{t_1,t_2}((C,L_1, L_2 ), \lambda)
&=
\mu(C, \lambda) +t_1\mu(L_1, \lambda)+t_2\mu(L_2, \lambda)
\\
&=
\min\{ 
3w_1i_0+3w_2i_1 \: | \; x_0^{i_0}x_1^{i_1}
\in Supp(f)
\}
+
(d+t_1+t_2)(-w_1-w_2)
\\
&=
3 \vec w(f)-(w_1+w_2)(d+t_1+t_2)
=- (w_1+w_2) \left( 
t_1+t_2 - \frac{3\vec w(f_C)}{w_1+w_2} + d  \right).
\end{align*}
Therefore, $\mu_{t_1,t_2}(C,L_1, L_2, \lambda)  >0$ and the triple is destabilized by $\lambda$.

For the second statement, the lines $L_1$ and $L_2$ have equation $l_1(x_0,x_1)$ and $l_2(x_0,x_1) +x_2$, respectively, where $l_1,l_2$ are linear forms. If $t_2-t_1\frac{2w_2-w_1}{w_1+w_2}-\frac{3\vec w(f_C)}{w_1+w_2}+d$, as $w_1\geqslant w_2$, we have
\begin{align*}
\mu_{t_1,t_2}((C,L_1, L_2), \lambda)
&=
\mu(C, \lambda) +t_i\mu(L_i, \lambda)+t_k\mu(L_k, \lambda)
\\
&\geqslant
3 \vec w(f_C)
+
d(-w_1-w_2) +t_1(2w_2-w_1)+t_2(-w_1-w_2)
\\
&=
-(w_1+w_2)
\left(
t_2-
t_1\frac{2w_2-w_1}{w_1+w_2}
-\frac{3w(f_C)}{w_1+w_2}+d
\right)>0.
\end{align*}
\end{proof}

For the rest of the paper we consider tuples $(C, L_1, L_2)$ formed by a plane quartic $C$ and two lines $L_1, L_2\subset \mathbb P^2$. 
The following result will come useful:
\begin{lemma}[{Shah \cite[Section 2]{shah1980complete}, cf. \cite[Theorem 1.3]{laza2012ksba}}]
\label{lemma:sextic}
Let $Z$ be a plane sextic, and $X$ the double cover of 
$\mathbb{P}^2$ branched along $Z$. Then $X$ has semi-log canonical singularities if and only if $Z$ is semistable and the closure of the orbit of $Z$ does not contain the orbit of the triple conic. In particular, a sextic plane curve with simple singularities is stable.  
\end{lemma}

\begin{lemma}\label{lemma:ADE}
Let $\vec t =(1,1)$ and $(C, L_1, L_2)$ be a tuple such that the sextic $C+L_1+L_2$ is reduced. Then, $(C, L_1, L_2)$ is $\vec t$-(semi)stable if and only if the double cover $X$ of $\mathbb P^2$ branched at $C+L_1+L_2$ has at worst  simple singularities (respectively  simple elliptic or cuspidal singularities).
\end{lemma}
\begin{proof}
The sextic $Z\coloneqq C+L_1+L_2=\{f\cdot l_1\cdot l_2=0\}$ (where $f$ is a quartic curve and $l_1, l_2$ are distinct linear forms not in the support of $f$) cannot degenerate to a triple conic and it is reduced by hypothesis. By \lemmaref{lemma:sextic}, $Z$ is a GIT-semistable sextic curve if and only if $X$ has semi-log canonical (slc) singularities. The surface $X$ is normal, as $Z$ is reduced \cite[Proposition 0.1.1]{cossec-dolgachev-enriques1}. In particular $X\coloneqq\{w^2=f\cdot l_1\cdot l_2\}\subset \mathbb P(1,1,1,3)$ has hypersurface log canonical singularities away from the singular point $(0:0:0:1)\not\in X$, and by the classification of such singularities in \cite[Table 1]{liu2012two}, they can only consist of either simple, simple elliptic or cuspidal singularities. If $Z$ has only simple singularities then $Z$ is GIT-stable by \lemmaref{lemma:sextic}. Now suppose $Z$ is GIT-stable and reduced. By \cite[Theorem 1.3 and Remark 1.4]{laza2012ksba} a GIT-semistable plane sextic curve has either simple singularities or it is in the open orbit of a sextic containing a double conic or a triple conic in its support, contradicting the fact that $Z$ is reduced. Hence $Z$ has only simple singularities. The proof follows from \lemmaref{lemma:tuple}.
\end{proof}
\begin{remark}
Although, we will not discuss other polarizations. It is worth to notice that for $\vec t=( \epsilon, \epsilon )$ the stability is very similar to the one of plane quartics. In particular, if $C$ is a semistable quartic and $L_1$, $L_2$ are lines in general position. Then, the triple $(C, L_1, L_2)$ is stable. 
\end{remark}

Let $\mathcal R^s_{\vec t}$ and $\mathcal R^{ss}_{\vec t}$ be the set of $\vec t$-stable and $\vec t$-semistable tuples $(f,l_1,l_2)$, respectively. Let
$$\mathcal R_0:=\left \{(f,l_1,l_2)\; | \; \text{the sextic }
\left\{f\cdot l_1\cdot l_2=0\right\}\text{ is reduced and has at worst simple singularities} 
\right\}.$$
Let $\vec t=(1,1)$. By \lemmaref{lemma:sextic}, $\mathcal R_0\subseteq\mathcal R^s_{\vec t}\subseteq\mathcal R^{ss}_{\vec t}$. Let $\mathcal M_0\coloneqq\mathcal R_0/ \! \! / \operatorname{PGL}_{3}$, $\mathcal M^s(1,1)\coloneqq\mathcal R^s_{(1,1)}/ \! \! / \operatorname{PGL}_{3}$ and recall that $\overline {\mathcal M}(1,1)=\mathcal R^{ss}_{(1,1)}/ \! \! / \operatorname{PGL}_{3}$. Then $\mathcal M_0\subseteq\mathcal M^s(1,1)\subset\overline{\mathcal M}(1,1)$. We are interested in describing the compactification of $\mathcal M_0$ by $\overline{\mathcal M}(1,1)$. We use the notation in \cite{laza2012ksba}.
\begin{lemma}\label{lemma:Bd11}
The quotient $\overline {\mathcal M}(1,1)$ is the compactification of $\mathcal M_0$ by three points and six rational curves. The three points correspond to the closed orbit of tuples $(C,L_1,L_2)$ defined up to projective equivalence by the following tuples:
\begin{align*}
	[III(1)]&: 	& C&=\{(x_0x_2-x_1^2)^2=0\},	& L_1&=\{x_0=0\},	& L_2&=\{x_2=0\};\\
	[III(2a)]&:	& C&=\{x_1^2x_2^2=0\},			& L_1&=\{x_0=0\}, & L_2&=\{x_0=0\};\\
	[III(2b)]&:	& C&=\{x_0x_1^2x_2=0\},		& L_1&=\{x_0=0\}, & L_2&=\{x_2=0\}.
\end{align*}
The six rational curves correspond to the closed orbit of tuples $(C,L_1,L_2)$ defined up to projective equivalence by the following cases:
\begin{align*}
	[II(1)]&: 		& C&=\{(x_0x_2-x_1^2)(x_0x_2-a x_1^2)=0\},		& L_1&=\{x_0=0\},	& L_2&=\{x_2=0\};\\
	[II(2a1)]&:		& C&=\{x_1x_2(x_2-x_1)(x_2-a x_1) =0\},				& L_1&=\{x_0=0\}, & L_2&=\{x_0=0\};\\
	[II(2a2)]&:	& C&=\{x_0x_2(x_2-x_1)(x_2-a x_1)=0\},				& L_1&=\{x_0=0\}, & L_2&=\{x_1=0\};\\
	[II(2b1)]&:		& C&=\{x_0^2(x_2-x_1)(x_2-a x_1) =0\},				& L_1&=\{x_1=0\}, & L_2&=\{x_2=0\};\\
	[II(2b2)]&:	& C&=\{x_0x_2(x_2-x_1)(x_2-a x_1)=0\},				& L_1&=\{x_1=0\}, & L_2&=\{x_0=0\};\\
	[II(3)]&:			& C&=\{(x_0x_2-x_1^2)^2=0\},									& & 							& L_1&=\{x_1=0\},
	\end{align*}
	\begin{align*}
	& \qquad	\qquad	\qquad	\qquad	\qquad	\qquad	\qquad	\qquad	\qquad	\qquad	\qquad	& L_2=\{a x_0 -(a +1)x_1 +x_2=0\}.
\end{align*}
where $a \neq 0,1, \infty$.
\end{lemma}
\begin{proof} 
Let $\mathcal R^{''}=\mathcal R_{1,5,1}$, parametrising tuples $(g,l_1)$ up to multiplication by scalar where $g$ is a quintic homogeneous polynomial and $l_1$ is a linear form. As we have seen in \propositionref{proposition:big-diagram}, we have a morphism $\phi_2:\mathcal R^{ss}\rightarrow \mathcal R^{''ss}$ defined by $\phi_2:(f,l_1,l_2)\mapsto(f\cdot l_1, l_2)$,and an orbit $O$ of $\mathcal R^{ss}$ is closed if and only if the orbit $\phi_2(O)$ of $\mathcal R^{''ss}$ is closed.

Hence the points which compactify $\mathcal M_0$ into $\overline{\mathcal M}(1,1)$ corresponding to closed orbits of $\mathcal R^{ss}\setminus \mathcal R_0$ are mapped via $\phi_2$ onto points in $\overline\phi_2(\overline{\mathcal M}(1,1))\setminus \overline \phi_2(\mathcal M_0)$ corresponding to closed orbits in $\phi_2(\mathcal R^{ss})\setminus \phi_2(\mathcal R_0)$. Hence we just need to identify closed orbits in $\overline{\mathcal M})(1)_{1,5,1}\cap \mathrm{Im}(\overline \phi_2)$. Our result is a straight forward identification of these orbits in the classification of $\overline{\mathcal M})(1)_{1,5,1}$ in \cite[Proposition 3.22]{laza_n16}.
\end{proof}

\begin{center}
\begin{figure}[h!]
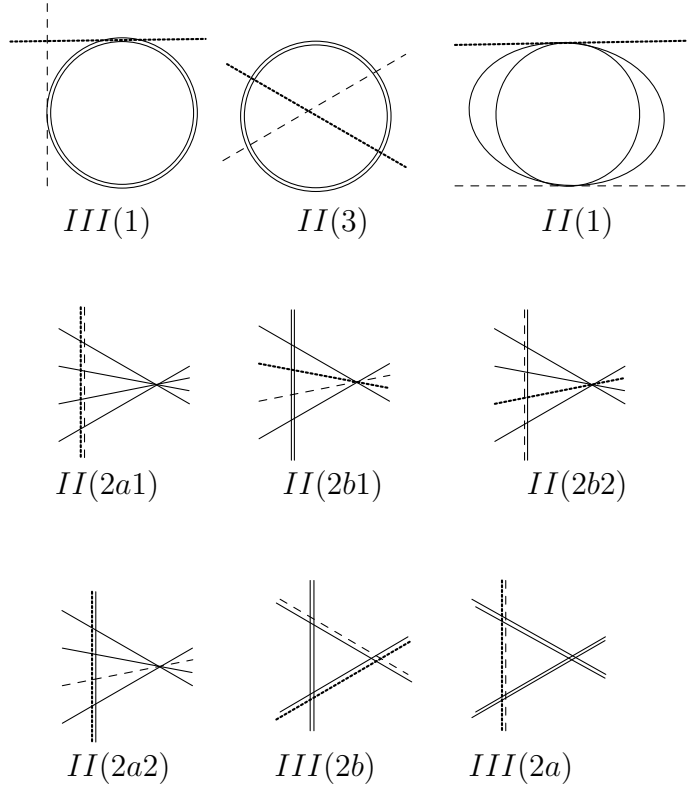



\definecolor{xdxdff}{rgb}{0,0,0}
\definecolor{qqqqff}{rgb}{0,0,0}
\tikzpicture[line cap=round,line join=round,>=triangle 45,x=1.0cm,y=1.0cm, scale=0.5]

\draw(-13.6,1.94) circle (1.9cm);
\draw(-13.6,1.94) circle (2cm);
\draw[dashed] (-15.59,0)-- (-15.59,5);
\draw[densely dotted, thick] (-16.57,3.84)-- (-11.38,3.91);
\draw (-14,-1) node {$III(1)$};

\draw(-8.45,1.85) circle (1.9cm); 
\draw(-8.45,1.85) circle (2cm); 
\draw (-8,-1) node {$II(3)$};


\draw[dashed] (-10.92,0.67)-- (-6.06,3.5);
\draw[densely dotted, thick] (-10.81,3.24)-- (-6.06,0.5);

\draw [rotate around={-6.03:(-1.78,1.91)}] (-1.78,1.91) ellipse (2.6cm and 1.89cm);
\draw [rotate around={78.82:(-1.75,1.9)}] (-1.75,1.9) ellipse (1.91cm and 1.91cm);
\draw[color=qqqqff] (-1.5,-1) node {$II(1)$};
\draw[dashed] (-4.75,0)-- (1.38,0);
\draw[densely dotted, thick]  (-4.75,3.75)-- (1.38,3.86);



\draw[dashed] (-14.6,-3.23)-- (-14.6,-7.23);
\draw[densely dotted, thick ](-14.7,-3.23)-- (-14.7,-7.23);

\draw (-10.28-5,-3.8)-- (-6.81-5,-5.8);
\draw (-10.28-5,-4.8)-- (-6.82-5,-5.45);
\draw (-10.28-5,-5.8)-- (-6.81-5,-5.11);
\draw (-10.28-5,-6.8)-- (-6.81-5,-4.8);
\draw (-14,-8) node {$II(2a1)$};

\draw (0.04-10,-3.73)-- (3.51-10,-5.73);
\draw[densely dotted, thick] (0.04-10,-4.73)-- (3.5-10,-5.38);
\draw[dashed] (0.04-10,-5.73)-- (3.52-10,-5.04);
\draw (0.04-10,-6.73)-- (3.51-10,-4.73);
\draw (0.9-10,-3.3)-- (0.9-10,-7.3);
\draw (0.98-10,-3.3)-- (0.98-10,-7.3);
\draw[color=qqqqff] (-8,-8) node {$II(2b1)$};

\draw (5.3-9,-3.8)-- (8.77-9,-5.8);
\draw (5.31-9,-4.8)-- (8.76-9,-5.45);
\draw[densely dotted, thick] (5.31-9,-5.8)-- (8.78-9,-5.11);
\draw (5.31-9,-6.8)-- (8.77-9,-4.8);
\draw (6.18-9,-3.3)-- (6.18-9,-7.3);
\draw[dashed] (6.10-9,-3.3)-- (6.10-9,-7.3);
\draw (-1.5,-8) node {$II(2b2)$};

\begin{scope}[xshift=-12cm,yshift=-14.5cm]

\draw (3.5,4)-- (3.5,0);
\draw (3.4,4)-- (3.4,0);
\draw [dashed] (2.6,3.5)-- (6.0,1.5);
\draw (2.5,3.4)-- (6.1,1.3);
\draw   (2.6,0.5)-- (6.0,2.5);
\draw[densely dotted, thick] (2.5,0.3)-- (6.1,2.4);
\draw[color=xdxdff] (3.8,-1) node {$III(2b)$};

\draw[dashed] (8.6,4)-- (8.6,0);
\draw[densely dotted, thick]  (8.5,4)-- (8.5,0);
\draw (7.8,0.5)-- (11.25,2.5);
\draw (7.7,0.35)-- (11.24,2.4);
\draw (7.7,3.5)-- (11.26,1.5);
\draw (7.8,3.3)-- (11.24,1.4);
\draw[color=xdxdff] (9,-1) node {$III(2a)$};
\end{scope}

\begin{scope}[yshift=-7.5cm,xshift=-17.5cm]
\draw (5.3-3,-3.8)-- (8.77-3,-5.8);
\draw (5.31-3,-4.8)-- (8.76-3,-5.45);
\draw[dashed] (5.31-3,-5.8)-- (8.78-3,-5.11);
\draw  (5.31-3,-6.8)-- (8.77-3,-4.8);
\draw (6.2-3,-3.3)-- (6.2-3,-7.3);
\draw[densely dotted, thick]  (6.10-3,-3.3)-- (6.10-3,-7.3);
\draw[color=qqqqff] (3.8,-8) node {$II(2a2)$};
\end{scope}

\endtikzpicture
\caption{Triples parametrized by $\overline {\mathcal M}(1,1)\setminus \mathcal M_0$. The dotted and dashed lines represent the lines $L_1$ and $L_2$, respectively (see \lemmaref{lemma:Bd11}).
}
\label{fig:Bd}
\end{figure}
\end{center}


\lemmaref{lemma:Bd11}, together with \propositionref{proposition:big-diagram} gives us the following compactification which will be of interest for the next section:
\begin{corollary} \label{cor:gitcll'}
Let  $\overline{\calM}^*\coloneqq\overline{\calM}^*_{1,4,2}$ be the GIT compactification of a quartic plane curve and two unlabeled lines and let $\mathcal M^*_0\subset\overline{\calM}^*$ be the open loci parametrizing  triples $(C,L,L')$ such that $C+L+L'$ is reduced and has at worst simple singularities. 
Then, $\overline{\calM}^* \setminus \mathcal M^*_0$  is the union of three points, $\overline{III}(1)$, $\overline{III}(2a)$, $\overline{III}(2b)$,	
and five rational curves, 
$\overline{II}(1)$, $\overline{II}(2a1)$, $\overline{II}(2a2)$, $\overline{II}(2b)$, $\overline{II}(3)$, 
which are obtained as images ---via the natural morphism $\pi_*\colon\overline{\calM}(1,1) \to \overline{\calM}^*$--- of points and rational curves described in \lemmaref{lemma:Bd11}, 
as follows:
\begin{itemize}
\item the points 
$\overline{III}(1)$, $\overline{III}(2a)$, $\overline{III}(2b)$ are the images of  the points 
$III(1)$, $III(2a)$, $III(2b)$, and 
\item 
the rational curves 
$\overline{II}(1)$, $\overline{II}(2a1)$, $\overline{II}(2a2)$, $\overline{II}(2b)$, $\overline{II}(3)$ are the images of 
the rational curves 
$II(1)$, $II(2a1)$, $II(2a2)$, $II(2b1)$, $II(3)$. 
\end{itemize}
Moreover,  the boundary components $II(2a2)$  and $II(2b2)$ in $\overline{\calM}(1,1)$ are mapped onto the same boundary component $\overline{II}(2a2)$ in $\overline{\calM}(1)$.
\end{corollary}

\section{Moduli of quartic plane curves and two lines via $K3$ surfaces}
\label{sec:Hodge}

\subsection{On K3 surfaces and lattices}  By a \emph{lattice} we mean a finite dimensional free $\bZ$-module $L$ together with a symmetric bilinear form $(-,-)$. The basic invariants of a lattice are its rank and signature. A lattice is \emph{even} if $(x,x) \in 2\bZ$ for every $x \in L$. The direct sum $L_1\oplus L_2$ of two lattices $L_1$ and $L_2$ is always assumed to be orthogonal, which will be denoted by $L_1\perp L_2$. For a lattice $M \subset L$, $M_L^{\perp}$ denotes the orthogonal complement of $M$ in $L$. Given two lattices $L$ and $L'$ and a lattice embedding $L \hookrightarrow L'$, we call it a \emph{primitive embedding} if $L'/L$ is torsion free.  

We shall use the following lattices: the (negative definite) root lattices $A_n$ ($n \geqslant 1$), $D_m$ ($m \geqslant 4$), $E_r$ ($r = 6,7,8$) and the hyperbolic plane $U$. Given a lattice $L$, $L(n)$ denotes the lattice with the same underlying $\bZ$-module as $L$ but with the bilinear form multiplied by $n$. 

\begin{notation}
Given any even lattice $L$, we define:
\begin{itemize}
\item $L^* := \{y \in L_{\bQ} := L \otimes \bQ \mid (x,y) \in \bZ \ \text{for all} \ x \in L\}$, the dual lattice;
\item $A_L := L^*/L$, the discriminant group endowed with the induced quadratic form $q_L: A_L \rightarrow \bQ/2\bZ$;
\item $\mathrm{disc}(L)$: the determinant of the Gram matrix (i.e. the intersection matrix) with respect to an arbitrary $\bZ$-basis of $L$;
\item $O(L)$: the group of isometries of $L$;
\item $O(q_L)$: the automorphisms of $A_L$ that preserve the quadratic form $q_L$;
\item $O_{-}(L)$: the group of isometries of $L$ of spinor norm $1$ (see \cite[\S 3.6]{scattone});
\item $\widetilde{O}(L)$: the group of isometries of $L$ that induce the identity on $A_L$;
\item $O^*(L) = O_{-}(L) \cap \widetilde{O}(L)$;
\item $\Delta(L)$: the set of roots of $L$ ($\delta \in L$ is a \emph{root} if $(\delta, \delta) = -2$);
\item $W(L)$: the \emph{Weyl group}, i.e. the group of isometries generated by reflections $s_{\delta}$ in root $\delta$, where $s_{\delta}(x) = x - 2\frac{(x,\delta)}{(\delta,\delta)}\delta$.
\end{itemize}
\end{notation}

For a surface $X$, the intersection form gives a natural lattice structure on the torsion-free part of $H^2(X, \bZ)$ and on the N\'{e}ron-Severi group $\mathrm{NS}(X)$. For a $K3$ surface $S$, we have $H^1(S, \calO_S) = 0$, and hence $\Pic(S)\cong \mathrm{NS}(S)$. Both $H^2(S, \bZ)$ and $\Pic(S)$ are torsion-free and the natural map $c_1: \Pic(S) \rightarrow H^2(S, \bZ)$ is a primitive embedding. Given any K3 surface $S$, $H^2(S,\bZ)$ is isomorphic to $\Lambda_{K3}\coloneqq E_8^{2} \perp U^{3}$, the unique even unimodular lattice of signature $(3,19)$. We shall use $O(S)$, $\Delta(S)$, $W(S)$, etc. to denote the corresponding objects of the lattice $\Pic(S)$. We also denote by $\Delta^+(S)$ and $V^+(S)$ the set of effective $(-2)$ divisor classes in $\Pic(S)$ and the K\"ahler cone of $S$ respectively.   


In our context, a \emph{polarization} for a $K3$ surface is the class of a nef and big divisor $H$ (and not the most restrictive notion of ample divisor, we follow the terminology in \cite{laza_n16}) and $H^2$ is its degree. More generally there is a notion of lattice polarization. We shall consider the period map for (lattice) polarized $K3$ surfaces and use the standard facts on $K3$ surfaces: the global Torelli theorem and the surjectivity of the period map. We also need the following theorem (see \cite[p. 40]{morrison_K3} or \cite[Theorem 4.8, Proposition 4.9]{laza_n16}).

\begin{theorem} \label{mayer}
Let $H$ be a nef and big divisor on a $K3$ surface $S$. The linear system $|H|$ has base points if and only if there exists a divisor $D$ such that $H \cdot D = 1$ and $D^2=0$. 
\end{theorem} 

\subsection{The K3 surfaces associated to a generic triple}
We first consider the $K3$ surfaces arising as a double cover of $\mathbb P^2$ branched at a smooth quartic curve $C$ and two different lines $L_1$ and $L_2$ such that $C+L_1+L_2$ has simple normal crossings. We shall show that these $K3$ surfaces are naturally polarized by a certain lattice.

Denote by $\bar{S}_{(C,L_1,L_2)}$ the double cover of $\bP^2$ branched along $C+L_1+L_2$. Let $S_{(C,L_1,L_2)}$ be the $K3$ surface obtained as the minimal resolution of the $9$ singular points of $\bar{S}_{(C,L_1,L_2)}$. Let $\pi: S_{(C,L_1,L_2)} \rightarrow \bP^2$ be the natural morphism. Note that $\pi: S_{(C,L_1,L_2)} \rightarrow \bP^2$ also factors as the composition of the blow-up of $\mathbb P^2$ at the singularities of $C+L_1+L_2$ and the double cover of the blow-up branched along the strict transforms of $C$, $L_1$ and $L_2$ (see \cite[\S III.7]{BPV}).

Let $h = \pi^*\calO_{\bP^2}(1)$ be the pullback of the class of a line in $\bP^2$. The class $h$ is a degree $2$ polarization of $S_{(C,L_1,L_2)}$. We assume that $C \cap L_1 = \{p_1, p_2, p_3, p_4\}$, $C \cap L_2 = \{q_1, q_2, q_3, q_4\}$ and $L_1 \cap L_2 =\{r\}$. Denote the classes of the exceptional divisors corresponding to $p_i$, $q_i$, and $r$ by $\alpha_i$, $\beta_i$ and $\gamma$ respectively ($1 \leqslant i \leqslant 4$). Let us also denote by $l_1'$ (respectively $l_2'$) the class of the strict transform of $L_1$ (respectively $L_2$). Note that the morphism $\pi: S_{(C,L)} \rightarrow \bP^2$ is given by the class
\begin{equation}
h = 2l_1' + \alpha_1 + \ldots + \alpha_4 + \gamma = 2l_2' + \beta_1 + \ldots + \beta_4 + \gamma.
\label{eq:class-h}
\end{equation}
It is straightforward to check that $(\alpha_i, \alpha_j) = (\beta_i, \beta_j) = -2 \delta_{ij}$, $(\gamma, \gamma) = -2$, $(\alpha_i, \beta_j) = (\alpha_i, \gamma) = (\gamma, \beta_j) = 0$ for $1 \leqslant i,j \leqslant 4$. Clearly, we have $(l_1', l_1') = (l_2',l_2') = -2$, $(l_1', \alpha_i) = (l_2', \beta_j) = (l_1', \gamma) = (l_2', \gamma)= 1$, and $(l_1', \beta_j) = (l_2', \alpha_i) = 0$ for $1 \leqslant i,j \leqslant 4$. 

Consider the sublattice of the Picard lattice of $S_{(C,L_1,L_2)}$ generated by the curve classes $\gamma$, $l_1'$, $\alpha_1, \ldots, \alpha_4$, $l_2'$, $\beta_1, \ldots, \beta_4$. Let $$\xi := 2l_1' + \alpha_1 + \ldots + \alpha_4 = 2l_2' + \beta_1 + \ldots + \beta_4.$$ It follows from \eqref{eq:class-h} that $\{ \gamma, l_1', \alpha_1, \alpha_2, \alpha_3, l_2', \beta_1, \beta_2, \beta_3, \xi\}$ forms a $\bZ$-basis of the sublattice. The Gram matrix with respect to this basis is computed as follows.
$$
 \left( {\begin{array}{cccccccccc}
   -2 & 1 & 0 & 0 & 0 & 1 & 0 & 0 & 0 & 2  \\
   1 & -2 & 1 & 1 & 1 & 0 & 0 & 0 & 0 & 0  \\
   0 & 1 & -2 & 0 & 0 & 0 & 0 & 0 & 0 & 0  \\
   0 & 1 & 0 & -2 & 0 & 0 & 0 & 0 & 0 & 0  \\
   0 & 1 & 0 & 0 & -2 & 0 & 0 & 0 & 0 & 0  \\
   1 & 0 & 0 & 0 & 0 & -2 & 1 & 1 & 1 & 0  \\
   0 & 0 & 0 & 0 & 0 & 1 & -2 & 0 & 0 & 0  \\
   0 & 0 & 0 & 0 & 0 & 1 & 0 & -2 & 0 & 0  \\
   0 & 0 & 0 & 0 & 0 & 1 & 0 & 0 & -2 & 0  \\
   2 & 0 & 0 & 0 & 0 & 0 & 0 & 0 & 0 & 0  \\
   \end{array} } \right)_{\textstyle .}
$$

\begin{notation} \label{lattice M}
Let $M$ be the abstract lattice of rank $10$ spanned by an ordered basis
$$\{ \gamma, l_1', \alpha_1, \alpha_2, \alpha_3, l_2', \beta_1, \beta_2, \beta_3, \xi\}$$
with the intersection form given by the above Gram matrix, which we will call $G_M$. Notice that $M$ is an even lattice. If $S_{(C,L_1,L_2)}$ is a $K3$ surface obtained as above from a smooth quartic $C$ and two lines $L_1,L_2$ such that $C+L_1+L_2$ has simple normal crossings, then there is a natural lattice embedding $\jmath: M\hookrightarrow \Pic(S_{(C,L_1,L_2)})$ as described before. 

We set $h = \gamma + \xi$. Observe that $\jmath(h)$ is linearly equivalent to the pullback of a line in $\mathbb P^2$ via $\pi$ and therefore, it is a base point free polarization. In particular, we have $(h,h) = 2$, $(h,l_1') = (h, l_2') = 1$, and $(h, \alpha_i) = (h, \beta_j) = 0$ for $1 \leqslant i,j \leqslant 3$. We also let $\alpha_4 = \xi - 2l_1' - \alpha_1 - \alpha_2 - \alpha_3$ and $\beta_4 = \xi - 2l_2' - \beta_1 - \beta_2 - \beta_3$.
\end{notation}



Let us compute the discriminant group $A_M$ and the quadratic form $q_M: A_M \rightarrow \bQ/2\bZ$. 
\begin{lemma} \label{A_M}
The discriminant group $A_M=M^*/M$ is isomorphic to $(\bZ/2\bZ)^{\oplus 6}$.
\end{lemma}
\begin{proof}
Let us denote by $\alpha_i^*\in M^*$ (respectively $\beta_j^*$, $\gamma^*$, $\xi^*$, $l_1'^*$, $l_2'^*\in M^*$) the dual element of $\alpha_i\in M$ (respectively $\beta_j$, $\gamma$, $\xi$, $l_1'$, $l_2'\in M$, for $1 \leqslant i,j \leqslant 3$). Recall that $\alpha_i^*$ is defined to be the unique element of $M^*$ such that $(\alpha_i^*, \alpha_i) = 1$ and the pairing of $\alpha_i^*$ with any other element of the basis $\{\gamma, l_1', \alpha_1, \alpha_2, \alpha_3, l_2', \beta_1, \beta_2, \beta_3, \xi\}$ is $0$. We define $l_1'^*$, $l_2'^*$, $\beta_j^*$, $\gamma^*$ and $\xi^*$ in a similar way. The coefficients of the dual elements $\gamma^*, l_1'^*, \alpha_1^*, \alpha_2^*, \alpha_3^*, l_2'^*, \beta_1^*, \beta_2^*, \beta_3^*, \xi^*$ (with respect to the basis $\{\gamma, l_1', \alpha_1, \alpha_2, \alpha_3, l_2', \beta_1, \beta_2, \beta_3, \xi\}$) can be read from the rows or columns of the inverse matrix $G_M^{-1}$ of the Gram matrix $G_M$ of $M$:
$$
G_M^{-1}= \left( {\begin{array}{cccccccccc}
   0 & 0 & 0 & 0 & 0 & 0 & 0 & 0 & 0 & \frac{1}{2}  \\
   0 & -2 & -1 & -1 & -1 & 0 & 0 & 0 & 0 & 1  \\
   0 & -1 & -1 & -\frac{1}{2} & -\frac{1}{2} & 0 & 0 & 0 & 0 & \frac{1}{2}  \\
   0 & -1 & -\frac{1}{2} & -1 & -\frac{1}{2} & 0 & 0 & 0 & 0 & \frac{1}{2}  \\
   0 & -1 & -\frac{1}{2} & -\frac{1}{2} & -1 & 0 & 0 & 0 & 0 & \frac{1}{2}  \\
   0 & 0 & 0 & 0 & 0 & -2 & -1 & -1 & -1 & 1  \\
   0 & 0 & 0 & 0 & 0 & -1 & -1 & -\frac{1}{2} & -\frac{1}{2} & \frac{1}{2}  \\
   0 & 0 & 0 & 0 & 0 & -1 & -\frac{1}{2} & -1 & -\frac{1}{2} & \frac{1}{2}  \\
   0 & 0 & 0 & 0 & 0 & -1 & -\frac{1}{2} & -\frac{1}{2} & -1 & \frac{1}{2} \\
   \frac{1}{2} & 1 & \frac{1}{2} & \frac{1}{2} & \frac{1}{2} & 1 & \frac{1}{2} & \frac{1}{2} & \frac{1}{2} & -\frac{1}{2}  \\
   \end{array} } \right)_{\textstyle .}
$$
For instance, $l_2'^*=-2l_2-\beta_1-\beta_2-\beta_3-\xi$, where we identify each element of $M$ with its image in $M^*$. By abuse of notation, we also use $\alpha_i^*$, $\beta_j^*$, $\gamma^*$, $\xi^*$ to denote the corresponding elements in $A_M = M^*/M$. Observe that $l_1'^* = l_2'^* \equiv 0 \in A_M$. It is straightforward to verify that $A_M$ can be generated by $\{\gamma^*, \alpha_1^*, \alpha_2^*, \beta_1^*, \beta_2^*, \xi^*\}$ and hence $A_M$ is isomorphic to $(\bZ/2\bZ)^{\oplus 6}$. Indeed, this follows from observing from the columns of $G_M^{-1}$ that $\alpha^*_3=\gamma^* + \alpha_1^*+\alpha_2^*\in A_M$ and $\beta^*_3=\gamma^* + \beta_1^*+\beta_2^*\in A_M$. 
\end{proof}

\begin{remark} \label{q_M}
We derive a formula for the quadratic form $q_M: A_M \rightarrow \bQ/2\bZ$:
$$
q_M(a\gamma^* + b\alpha_1^* + c\alpha_2^* + d\beta_1^* + e\beta_2^* + f\xi^*) \equiv b^2 + c^2 +bc + d^2 + e^2 + de + (a+b+c+d+e)f - \frac12 f^2 \in \bQ/2\bZ.
$$
\end{remark}

\begin{proposition} \label{primitive_irreducible}
Let $S$ be a $K3$ surface. If $\jmath: M\hookrightarrow \Pic(S)$ is a lattice embedding such that $\jmath(h)$ is a base point free polarization and $\jmath(l_1')$, $\jmath(l_2')$, $\jmath(\alpha_i)$, and $\jmath(\beta_j)$ ($1 \leqslant i,j \leqslant 3$) all represent irreducible curves, then $\jmath$ is a primitive embedding.
\end{proposition}
\begin{proof}
Assume that $\jmath$ is not primitive. Then the embedding $\jmath$ must factor through the saturation $\mathrm{Sat}(M)$ of $M$ which is a non-trivial even overlattice of $M$: $M \subsetneq \mathrm{Sat}(M) \hookrightarrow \Pic(S)$. By \cite[Proposition 1.4.1]{nikulin}, there is a bijection between even overlattices of $M$ and isotropic subgroups of $A_M=M^*/M$ (which are generated by isotropic elements, i.e. $v\in A_M$ such that $q_M(v)=0$). Using \lemmaref{A_M} and \remarkref{q_M}, it is easy to classify the isotropic elements of $A_M$. As $\alpha_1^* + \alpha_2^* +\gamma^*= \alpha_3^*$ and $\beta_1^* + \beta_2^* +\gamma^*= \beta_3^*$ in $A_M$, there are only three cases to consider. We drop the embedding $\jmath$ in the rest of the proof.
\begin{itemize}
\item{(Case 1)} The isotropic element is $\gamma^*$. From the columns of $G_M^{-1}$ we see that $\gamma^* = \frac12 \xi \in A_M$. Hence, we have $\xi = 2x$ for some $x \in \Pic(S)$. But then $(x, x) =\frac{1}{2}(\xi,\xi)= 0$ and $(h,x) =\frac12(h, \xi)= 1$ which would imply that $h$ is not base point free by \theoremref{mayer}.
\item{(Case 2)} The isotropic element is $\alpha_i^* + \beta_j^*$ where $1 \leqslant i,j \leqslant 3$. Let us take $\alpha_1^* + \beta_1^*$ for example. The other cases are similar. Note that $\alpha_1^* + \beta_1^* = - \frac12 \alpha_2 -\frac12 \alpha_3 -\frac12 \beta_2 - \frac12 \beta_3$ in $A_M$. We have $\alpha_2 + \alpha_3 + \beta_2 + \beta_3 = 2y$ for some $y \in \Pic(S)$. Because $S$ is a K3 surface and $(y,y)=-2$, either $y$ or $-y$ is effective. Note that $l_1'$, $l_2'$, $\alpha_i$ and $\beta_j$ ($1 \leqslant i,j \leqslant 3$) are  irreducible curves (by the assumption), $h$ is nef and $(2h+l_1',l_1')=0$. 
It follows that $2h+l_1'$ is nef. Since $(2h+l_1', y) = 1$, $y$ is effective. Because $(y,\alpha_2) = (y, \alpha_3) = (y, \beta_2) = (y, \beta_3) = -1$, we know $\alpha_2$, $\alpha_3$, $\beta_2$ and $\beta_3$ are in the support of $y$. Write $y = m\alpha_2 + n\alpha_3 + k \beta_2 + l \beta_3 + D = \frac12(\alpha_2 + \alpha_3 + \beta_2 + \beta_3)$ where $D$ is an effective divisor, $\alpha_2,\alpha_3,\beta_2,\beta_3\not\subset\mathrm{Supp}(D)$ and $m,n,k,l \geqslant 1$. But then we have a contradiction
$$1=(y, 2h+l_1') \geqslant m+n\geqslant 2.$$
\item{(Case 3)} The isotropic element is $\alpha_i^* + \beta_j^* + \gamma^*$ where $1 \leqslant i,j \leqslant 3$. Take $\alpha_1^* + \beta_1^* + \gamma^*$ for example. Since $\alpha_1^* + \beta_1^* + \gamma^* = \frac12 \alpha_2 + \frac12 \alpha_3 + \frac12 \beta_2 + \frac12 \beta_3 + \frac12 \xi$ in $A_M$, there exists an element $z$ of $\Pic(S)$ such that $2z = \alpha_2 +  \alpha_3 +  \beta_2 +  \beta_3 +  \xi$. Because $S$ is a K3 surface, $(z,z) = -2$ and $(z,h) = 1$, the class $z$ represents an effective divisor. By the assumption $l_1'$, $l_2'$, $\alpha_i$ and $\beta_j$ ($1 \leqslant i,j \leqslant 3$) represent irreducible curves. Note that $(z, \alpha_2) = (z, \alpha_3) = (z, \beta_2) = (z, \beta_3) = -1 < 0$. Let us write $z = m\alpha_2 + n\alpha_3 + k\beta_2 + l\beta_3 + D$, where $D$ is effective, $\alpha_2,\alpha_3,\beta_2,\beta_3\not\subset \mathrm{Supp}(D)$ and $m,n,k,l > 0$. Then we have $$2D = (1-2m)\alpha_2 + (1-2n)\alpha_3 + (1-2k)\beta_2 + (1-2l)\beta_3$$ which implies that $(D, l_1') < 0$ and $(D, l_2') < 0$. Now we write $$z = m\alpha_2 + n\alpha_3 + k\beta_2 + l\beta_3 + sl_1' + tl_2'+D'$$ where $D'$ is effective,  $\alpha_2,\alpha_3,\beta_2,\beta_3, l_1',l_2'\not\subset \mathrm{Supp}(D)$ and $m,n,k,l,s,t \geqslant 1$. But this is impossible: $2\leqslant s+t \leqslant (z, h) = 1$.
\end{itemize}
\end{proof}


\begin{corollary} \label{generic primitive}
Let $C$ be a smooth plane quartic curve and $L_1$, $L_2$ two distinct lines such that $C+L_1+L_2$ has simple normal crossings and let $\jmath: M\hookrightarrow \Pic(S_{(C,L_1,L_2)})$ be the lattice embedding given in \notationref{lattice M}. Then $\jmath$ is a primitive embedding.
\end{corollary}

The proof of \propositionref{primitive_irreducible} can easily be adapted to proof the following Lemma.
\begin{lemma} \label{primitive}
Let $S$ be a $K3$ surface and $\jmath: M\hookrightarrow \Pic(S)$ be a lattice embedding. If none of $\jmath(\xi)$, $\jmath(\alpha_i + \alpha_{i'} + \beta_j + \beta_{j'})$ or $\jmath(\alpha_i + \alpha_{i'} + \beta_j + \beta_{j'} + \xi)$ ($1 \leqslant i, i', j, j' \leqslant 3$) is divisible by $2$ in $\Pic(S)$, then the embedding $\jmath$ is primitive.
\end{lemma}

\begin{proposition} \label{generic surjectivity}
Assume that $S$ is a $K3$ surface such that $\Pic(S)$ is isomorphic to the lattice $M$. Then $S$ is the double cover of $\bP^2$ branched over a reducible curve $C + L_1 + L_2$ where $C$ is a smooth plane quartic, $L_1, L_2$ are lines and $C+L_1+L_2$ has simple normal crossings.
\end{proposition}
\begin{proof}
By assumption there exist $h, \gamma, l_1', \alpha_1, \ldots, \alpha_4, l_2', \beta_1, \ldots, \beta_4\in \Pic(S)$ satisfying the numerical conditions in \notationref{lattice M}. Without loss of generality, we assume that $h$ is nef (this can be achieved by acting by $\pm W(S)$). Then $l_1'$ and $l_2'$ are both effective (as $(l_i',l_i')=-2$, $(h,l_i')=1$). We further assume that $\alpha_i$, $\beta_j$ ($1 \leqslant i,j \leqslant 4$) and $\gamma$ are effective (apply $s_{\alpha_i}$ or $s_{\beta_j}$ or $s_{\gamma}$ if necessary).  

As $h$ is nef and $(h,h)=2>0$, $h$ is a polarization of degree $2$. We will show that $h$ is base point free by \emph{reductio ad absurdum}. By \theoremref{mayer}, there exists a divisor $D$ such that $(D, D) = 0$ and $(h, D) = 1$. Note that this is a numerical condition. Write $D$ as a linear combination of $\gamma, l_1', \alpha_1, \ldots, \alpha_3, l_2', \beta_1, \ldots, \beta_3$ and $\xi$, with coefficients $c_1,\ldots, c_{10}$. Let   $S_{(Q,L,L')}$ be the $K3$ surface associated to a smooth quartic curve $Q$ and two lines $L$, $L'$ such that $Q+L+L'$ has simple normal crossings. Find the curve classes corresponding to $\gamma, l_1', \alpha_1, \ldots, \alpha_3, l_2', \beta_1, \ldots, \beta_3,\xi$ (as what we did at the beginning of this subsection) and consider their linear combination $D'$ with coefficients $c_1,\ldots, c_{10}$, the same values as in the expression for $D$. Then, both $D$ and $D'$ satisfy the same numerical conditions in $S$ (respectively $S_{(Q,L,L')}$) with respect to the divisor class $h= \gamma + \xi$. Again, by \theoremref{mayer} the pull-back $h$ of $\sim \calO_{\bP^2}(1)$ in $S_{(Q,L,L')}$ has base points, which gives a contradiction. So the linear system of $h$ defines a degree two map $\pi:S \rightarrow \bP^2$. Since $S$ is a $K3$ surface of degree $2$, the branching locus must be a sextic curve $B$.

Consider $h' = 3h + l_1' + l_2'$. Note that $(h', h')>0$ and $(h', h) > 0$. We can write any effective divisor as
\begin{equation}
D=a_0l_1'+a_1\alpha_1+\cdots+a_4\alpha_4+b_0l_2'+b_1\beta_1+\cdots+b_4\beta_4+c\gamma,
\label{eq:basis_D}
\end{equation}
where $a_i, b_i, c\in \mathbb Z$, where $0\leqslant i\leqslant 4$. Let $k_i=\frac{1}{2}a_0-a_i$, $l_i=\frac{1}{2}b_0-b_i$ where $1\leqslant i\leqslant 4$. It follows that 
\begin{align}
(D,h)&=a_0+b_0, \qquad (D,l_1')=c-2a_0+\sum_{i=1}^4a_i, \qquad (D, l_2')= c-2b_0+\sum_{i=1}^4b_i, \label{eq:properties_D1}\\
(D,D)&=-2\left(\sum_{i=1}^4k_i^2 + \sum_{i=1}^4l_i^2\right)+2c(a_0+b_0-c). \label{eq:properties_D2}
\end{align}
Let $D\in \Delta(S)$ as in \eqref{eq:basis_D}. Then $(D,D)=-2$ implies that
\begin{equation}
c^2+\sum_{i=1}^4\left(k_i^2+l_i^2\right)=1+c(a_0+b_0).
\label{eq:properties_D3}
\end{equation}
First note that when $a_0+b_0=(D,h)=0$, then \eqref{eq:properties_D3} gives that either $c=\pm 1$ and $D=\pm \gamma$ or 
$c=0$, and all coefficients in $\{k_1,\ldots, k_4, l_1,\ldots, l_4\}$ but one equal $0$ and $D\in \{\pm \alpha_1,\ldots, \pm\alpha_4, \pm\beta_1,\cdots, \pm\beta_4\}$. In particular, $\langle h \rangle^{\perp}_M \cap \Delta(S) = \{\pm \gamma, \pm \alpha_1, \ldots, \pm \alpha_4, \pm \beta_1, \ldots, \pm \beta_4\}$. If $D\in \Delta^+(S)\cap \langle h \rangle^{\perp}_M$, then $(D,h)=a_0+b_0=0$ which in turn implies that $(h',D)=(l_1'+l_2',D)>0$.

Now suppose that $D\in \Delta^+(S)$ and $(h,D)>0$. Then \eqref{eq:properties_D1} implies that $a_0+b_0>0$ and \eqref{eq:properties_D3} gives  $c\geqslant 0$. Then, by the arithmetic-geometric mean inequality and \eqref{eq:properties_D3}, we get
\begin{align*}
(h',D)&=a_0+\cdots +a_4-2a_0+2b_0+\cdots b_4+2c=3a_0+3b_0+2c-\sum_{i=1}^4\left(k_i+l_i\right)\\
&\geqslant (3(a_0+b_0)+2c)-2\sqrt{2}\sqrt{1+c(a_0+b_0-c)}>0,
\end{align*}
where the latter inequality follows from observing that the first summand is positive and
\begin{align*}
(3a_0+3b_0+2c)^2-\left(2\sqrt{2}\sqrt{1+c(a_0+b_0-c)}\right)^2=9(a_0+b_0)^2+12c^2+4c(a_0+b_0)- 8\geqslant 1>0.
\end{align*}
Hence $(h',D)>0$ for all $D\in \Delta^+(S)$. Moreover, if $D \subset S$ is rational and $ D \not\in \Delta^+(S)$, then $\pi_*(D)\neq 0$ and $(h',D)=(\pi_*(h),\pi_*(D))=\deg(\pi_*(D))>0$. Hence, by \cite[Cor. 8.1.7]{huybrechts_k3}, $h'$ is ample. 

Because $(h', l_1')=1$, the class $l_1'$ is represented by an irreducible curve. Similarly, $l_2'$, $\alpha_i$ and $\beta_j$ ($1 \leqslant i,j \leqslant 3$) all correspond to irreducible curves. It follows that the irreducible rational curves $\alpha_1, \ldots,  \alpha_4, \beta_1, \ldots, \beta_4$ 
are contracted by $\pi$ to ordinary double points of the sextic $B$. Let $L_1'$ (respectively $L_2'$) be the unique irreducible curve in $S$ corresponding to the class $l_1'$ (respectively $l_2'$) and set $L_1 = \pi(L_1')$ (respectively $L_2 = \pi(L_2')$). Since $(l_1', h) = 1$, the projection formula implies that $L_1$ is a line. Moreover, the line $L_1$ has to pass through four ordinary double points of the branched curve $B$ since $(l_1', \alpha_1) = \ldots = (l_1', \alpha_4) = 1$. Similarly, $L_2$ is also a line passing through four different ordinary double points of $B$. (Note that both $L_1$ and $L_2$ pass through the singularity of $B$ corresponding to $\gamma$.) By Bezout's theorem, the two lines $L_1$ and $L_2$ are both components of $B$ (otherwise we have contradictions: $(L_1,B) = (l_1', \pi^*(B)) \geqslant \sum_{p_i\in B\cap L_1}\mathrm{mult}_{p_i}(B)\geqslant 4 \cdot 2 > 6$ and analogously for $L_2$). 
\end{proof}

\begin{corollary}  \label{generic Picard}
For a sufficiently general triple $(C, L_1, L_2)$ (i.e. outside the union of a countable number of proper subvarieties of the moduli space), the Picard lattice $\Pic(S_{(C,L_1,L_2)})$ coincides with $M$ via the embedding $\jmath$. 
\end{corollary}
\begin{proof}
The argument in \cite[Corollary 6.19]{LazaThesis} works here. Alternatively, let $L_1$ and $L_2$ be given by linear forms $l_1$ and $l_2$, respectively, and consider the elliptic fibration $S_{(C,L_1,L_2)} \rightarrow \bP^1$ defined by the function $\pi^*(l_1/l_2)$. If $(C, L_1, L_2)$ is sufficiently general, then the pencil of lines generated by $L_1$ and $L_2$ only consists of lines intersecting $C$ normally or lines tangent to $C$ at a point. As a result, the elliptic fibration contains $2$ reducible singular fibers of type $I_0^*$ (i.e. with $5$ components) and $12$ singular fibers of type $I_1$ (i.e. with one nodal component), where we follow Kodaira's notation as in \cite[\S V.7]{BPV} \cite[\S 11.1.3]{huybrechts_k3}. Note that the fibration admits a 2-section $\gamma$. Consider the associated Jacobian fibration $J(S_{(C,L_1,L_2)})) \rightarrow \bP^1$ (see for example \cite[\S 11.4]{huybrechts_k3}). By the Shioda-Tate formula \cite[Corollary 11.3.4 and Corollary 11.4.7]{huybrechts_k3}, the $K3$ surface $S_{(C,L_1,L_2)}$ has Picard number $10$ which equals the rank of $M$. Moreover, we have \cite[\S 11 (4.5)]{huybrechts_k3}:
$$\mathrm{disc}(\Pic(S_{(C,L_1,L_2)})) = 2^2 \cdot \mathrm{disc}(\Pic(J(S_{(C,L_1,L_2)}))) = 64.$$
It is easy to compute that the Gram matrix $G_M$ has determinant $(-64)$. The proposition then follows from the following standard fact on lattices (which implies $[\Pic(S_{(C,L_1,L_2)}): M]=1$):
$$\mathrm{disc}(M) = \mathrm{disc}(\Pic(S_{(C,L_1,L_2)})) \cdot [\Pic(S_{(C,L_1,L_2)}): M]^2. $$
As $M\hookrightarrow  (\Pic(S_{(C,L_1,L_2)})$ and they have the same rank and discriminant, then $M\cong  (\Pic(S_{(C,L_1,L_2)})$.
 \end{proof}

Now let us consider the case when $C$ has at worst simple singularities not contained in $L_1+L_2$ and $C+L_1+L_2$ has simple normal crossings away from the singularities of $C$. We still use $S_{(C,L_1,L_2)}$ to denote the $K3$ surface obtained as a minimal resolution of the double cover of $\bP^2$ along $C+L_1+L_2$. The rank $10$ lattice $M$ is the same as in \notationref{lattice M}.
\begin{lemma}
If $C$ has at worst simple singularities not contained in $L_1+L_2$ and $C+L_1+L_2$ has simple normal crossings away from the singularities of $C$, then there exists a primitive embedding $\jmath: M \hookrightarrow \Pic(S_{(C, L_1, L_2)})$ such that $\jmath(h)$ is a base point free degree two polarization. 
\end{lemma}
\begin{proof}
Thanks to the transversal intersection, we define the embedding $\jmath$ as in the generic case. In particular, the morphism $\pi: S_{(C, L_1, L_2)} \rightarrow \bP^2$ is defined by $\jmath(h)$. The embedding $\jmath$ is primitive by \propositionref{primitive_irreducible}.
\end{proof}

\subsection{$M$-polarized $K3$ surfaces and the period map} \label{sec:MK3}

In this subsection let us compute the (generic) Picard lattice $M$ and the transcendental lattice $T$. Then we shall determine the period domain $\calD$ and define a period map for generic triples $(C, L_1, L_2)$ via the periods of $M$-polarized $K3$ surfaces $S_{(C, L_1, L_2)}$.

\begin{definition} \label{M polarized}
Let $M$ be the lattice defined in \notationref{lattice M}. An \emph{$M$-polarized $K3$ surface} is a pair $(S, \jmath)$ such that $\jmath: M\hookrightarrow \Pic(S)$ is a primitive lattice embedding. The embedding $\jmath$ is called the \emph{$M$-polarization} of $S$. We will simply say that $S$ is an \emph{$M$-polarized $K3$ surface} when no confusion about $\jmath$ is likely.
\end{definition}

We now determine the lattice $M$ and show that it admits a unique primitive embedding into the $K3$ lattice $\Lambda_{K3}$.
\begin{proposition} \label{determine M T}
Let $M$ be the lattice defined in \notationref{lattice M}. Then $M$ is isomorphic to the lattice $U(2) \perp A_1^{2} \perp D_6$ and admits a unique primitive embedding (up to isometry) $M \hookrightarrow \Lambda_{K3}$ into the $K3$ lattice $\Lambda_{K3}$. The orthogonal complement $T := M_{\Lambda_{K3}}^{\perp}$ with respect to the embedding is isometric to $U \perp U(2) \perp A_1^{2} \perp D_6$.
\end{proposition}
\begin{proof}
By \cite[Corollary 1.13.3]{nikulin} the lattice $M$ is uniquely determined by its invariants which can be easily computed from the Gram matrix $G_M$ (see also \lemmaref{A_M} and \remarkref{q_M}). 
\begin{itemize}
\item $M$ has rank $10$ and signature $(1,9)$.
\item The Gram matrix $G_M$ has determinant $(-64)$.
\item The discriminant group is $A_M \cong (\bZ/2\bZ)^{\oplus 6}$ with quadratic form $q_M = u \oplus w_{2,1}^{1} \oplus w_{2,1}^{1} \oplus w_{2,1}^{-1} \oplus w_{2,1}^{-1} $, where $u$, $w_{2,1}^{1}$ and $w_{2,1}^{-1}$ are the discriminant forms associated to $U(2)$, $E_7$ and $A_1$ respectively (cf. \cite[\S1.5 \& Appendix A]{belcastro} and references therein). Note that $w_{2,1}^{1} \oplus w_{2,1}^{1}$ is isomorphic to the discriminant form of $D_6$. 
\end{itemize}
By \cite[Theorem 1.14.4]{nikulin} the lattice $M$ admits a unique primitive embedding into $\Lambda_{K3}$. The claim on the orthogonal complement $T$ follows from \cite[Proposition 1.6.1]{nikulin}. 
\end{proof}

\begin{remark}
Note that both $M$ and $T$ are even indefinite $2$-elementary lattices (a lattice $L$ is $2$-elementary if $L^*/L\cong \left(\mathbb Z/2\mathbb Z\right)^k$ for some $k)$. One could also invoke Nikulin's classification \cite[Theorem 3.6.2]{nikulin} of such lattices to prove the previous proposition.  Moreover, $M$ and $T$ are orthogonal to each other in a unimodular lattice and hence $(A_M, q_M) \cong (A_T, -q_T)$, so $O(q_M) \cong O(q_T)$. 
\end{remark}

The moduli space of $M$-polarized $K3$ surfaces is a quotient $\calD/\Gamma$ for a certain Hermitian symmetric domain $\calD$ of type IV and some arithmetic group $\Gamma$ (see \cite{dolgachev_MSK3}). Fix the (unique) embedding $M \hookrightarrow \Lambda_{K3}$ and define
\begin{equation}
\calD = \{\omega \in \bP(\Lambda_{K3} \otimes \bC) \;|\; (\omega, \omega) = 0, (\omega, \bar{\omega}) > 0, \omega \perp M\}_0
\label{eq:period-domain1}
\end{equation}
to be one of the two connected components. Note that $\calD$ can also be identified with 
\begin{equation}
\calD\cong \{\omega \in \bP(T \otimes \bC) \;|\; (\omega, \omega) = 0, (\omega, \bar{\omega}) > 0\}_0.
\label{eq:period-domain2}
\end{equation}

To specify the moduli of $M$-polarized $K3$ surfaces, one also needs to determine the arithmetic group $\Gamma$. In the standard situation considered in \cite{dolgachev_MSK3} it is required that the $M$-polarization is pointwise fixed by the arithmetic group and one takes $\Gamma$ to be $O^*(T)$. In our geometric context the choice is different. Specifically, the permutations among $\alpha_1, \ldots, \alpha_4$ and among $\beta_1, \ldots, \beta_4$ are allowed. Observe that at the moment we view the lines $L_1$ and $L_2$ as labeled lines, distinguishing the tuples $(C, L_1, L_2)$ and $(C, L_2, L_1)$ and we do not consider the isometry of $M$ induced by flipping the two lines. 

Let $L$ be an even lattice. Recall that any $g \in O(L)$ naturally induces $g^* \in O(L^*)$ by $g^*\varphi: v \mapsto \varphi(g^{-1}v)$ (which further defines an automorphism of $A_L$ preserving $q_L$, therefore giving a natural homomorphism $r_L\colon O(L)\rightarrow O(q_L)$). 
\begin{lemma}
The homomorphisms $r_M: O(M) \rightarrow O(q_M)$ and $r_T: O(T) \rightarrow O(q_T)$ are both surjective. 
\end{lemma}
\begin{proof}
The lemma follows from \lemmaref{A_M} and \cite[Theorem 1.14.2]{nikulin}. 
\end{proof}

In particular, we have $O(M) \twoheadrightarrow O(q_M) \cong O(q_T) \twoheadleftarrow O(T)$. By \cite[Theorem 1.6.1, Corollary 1.5.2]{nikulin}, an automorphism $g_M \in O(M)$ can be extended to an automorphism of $\Lambda_{K3}$ if and only if $r_M(g_M) \in \im(r_T)$. In our case, any automorphism $g_M \in O(M)$ can be extended to an element in $O(\Lambda_{K3})$. 

\begin{lemma} \label{g_T extension}
Let $g_M$ (respectively $g_T$) be an automorphism of $M$ (respectively $T$). If $r_M(g_M) = r_T(g_T)$, then $g_M$ can be lifted to $g \in O(\Lambda_{K3})$ with $g|_T = g_T$. The same statement holds for $g_T$.
\end{lemma}
\begin{proof}
The proof is similar to that for \cite[Prop. 14.2.6]{huybrechts_k3}. Take any $x = x_M + x_T \in \Lambda_{K3}$ with $x_M \in M^*$ and $x_T \in T^*$. View $\Lambda_{K3}$ as an overlattice of $M \perp T$. The corresponding isotropic subgroup (cf. \cite[\S 1.4]{nikulin}) of $A_{M \perp T} \cong A_M \oplus A_T$ is $\Lambda_{K3}/(M \perp T)$. Since $x \in \Lambda_{K3}$, $\bar{x}_M + \bar{x}_T$ is contained in $\Lambda_{K3}/(M \perp T)$ (where $\bar{x}_M$ denotes the corresponding element of $x_M$ in $A_M$ and similarly for $\bar{x}_T$). Consider $g_M(x_M) + g_T(x_T) \in M^* \oplus T^*$. Note that the image of $g_M(x_M) + g_T(x_T)$ under the map $M^* \oplus T^* \rightarrow A_{M \perp T} \cong A_M \oplus A_T$ is $r_M(g_M)(\bar{x}_M) + r_T(g_T)(\bar{x}_T)$. Recall that $A_M$ and $A_T$ are identified via the natural projections $A_M \stackrel{\sim}{\leftarrow} \Lambda_{K3}/(M \perp T) \stackrel{\sim}{\rightarrow} A_T$. Because $r_M(g_M) = r_T(g_T)$, $r_M(g_M)(\bar{x}_M) + r_T(g_T)(\bar{x}_T)$ is contained in $\Lambda_{K3}/(M \perp T)$. In other words, we have $g_M(x_M) + g_T(x_T) \in \Lambda_{K3}$. 
\end{proof}

Let $\Sigma_{\alpha} \subset O(M)$ (respectively $\Sigma_{\beta} \subset O(M)$) be the subgroup which permutes $\{\alpha_1, \ldots, \alpha_4\}$ (respectively the subgroup which permutes $\{\beta_1, \ldots, \beta_4\}$). We seek automorphisms of $T$ which can be extended to automorphisms of $\Lambda_{K3}$ whose restrictions to $M$ belong to $\Sigma_\alpha$ or $\Sigma_\beta$. We observe that there is a natural inclusion $\Sigma_{\alpha} \times \Sigma_{\beta} \hookrightarrow O(M)$.
\begin{lemma} \label{permutation}
The composition $\Sigma_\alpha \times\Sigma_{\beta} \hookrightarrow O(M) \twoheadrightarrow O(q_M)$ is injective. 
\end{lemma}
\begin{proof}
First let us describe the automorphisms of $A_M$ induced by the transpositions in $\Sigma_\alpha$ and $\Sigma_{\beta}$. We consider $\Sigma_\alpha$ and the case of $\Sigma_\beta$ is analogous. The image of the transposition $(\alpha_i\alpha_i')$ (with $1 \leqslant i \neq i' \leqslant 3$) defines the element $r_M((\alpha_i\alpha_i'))$ in $O(q_M)$ given by $\alpha_i^* \mapsto \alpha_{i'}^*$, $\alpha_{i'}^* \mapsto \alpha_i^*$, leaving $\gamma^*$, $\alpha_{i''}^*$ ($i'' \in \{1,2,3\} \backslash \{i, i'\}$), $\beta_j^*$ ($1 \leqslant j \leqslant 3$) and $\xi^*$ invariant. 

The automorphism of $A_M$ induced by the transposition $(\alpha_1\alpha_4)$ between $\alpha_1$ and $\alpha_4$
is given by
\begin{align*}
&&\alpha_2^* \mapsto \alpha_3^* + \gamma^* \equiv\alpha_1^* + \alpha_2^*,  &&\alpha_3^* \mapsto \alpha_2^* + \gamma^* \equiv\alpha_1^* + \alpha_3^*, &&\xi^* \mapsto \xi^* + \alpha_1^*,
\end{align*}
and $\gamma^*$, $\alpha_1^*$, $\beta_j^*$ ($1 \leqslant j \leqslant 3$) are invariant by this action. The case of transpositions $(\alpha_2\alpha_4)$ and $(\alpha_3\alpha_4)$ is analogous. As it is well-known, the transpositions generate $\Sigma_\alpha$ and $\Sigma_\beta$. It is easy to compute the image of $\Sigma_\alpha \times \Sigma_{\beta}$ in $O(q_M)$ using the previous descriptions. 

Let $g_\alpha \in \Sigma_\alpha$ and $g_\beta \in \Sigma_\beta$. Now we describe how to univocally recover $(g_\alpha, g_\beta) \in \Sigma_\alpha \times\Sigma_{\beta}$ from the induced action $\bar{g}$ on $A_M$. In particular, this will show that the composed map $\Sigma_\alpha \times\Sigma_{\beta} \rightarrow O(q_M)$ is injective. Consider $\bar{g}(\xi^*)$. Because $q_M(\xi^*) \equiv -\frac12 \in \bQ/2\bZ$, the induced action $\bar{g}$ sends $\xi^*$ to an element $v^*$ satisfying $q_M(v^*) \equiv -\frac12$. By \remarkref{q_M} such elements are $\xi^*$, $\xi^* + \alpha_i^*$, $\xi^* + \beta_j^*$ and $\xi^* + \alpha_i^*+\beta_j^*$ ($1 \leqslant i,j \leqslant 3$). 
\begin{itemize}
\item If $\bar{g}(\xi^*) = \xi^*$, then $g_\alpha$ (respectively $g_\beta$) fixes $\alpha_4$ (respectively $\beta_4$) by the description of the permutations above and $g_\alpha$ (respectively $g_\beta$) can be recovered from the action of $\bar{g}$ on the set $\{\alpha_1^*, \alpha_2^*, \alpha_3^*\}$ (respectively $\{\beta_1^*, \beta_2^*, \beta_3^*\}$). 
\item If $\bar{g}(\xi^*) = \xi^* + \alpha_i^*$ ($1 \leqslant i \leqslant 3$), then $g_\alpha$ maps $\alpha_4$ to $\alpha_i$ and $g_\beta$ fixes $\beta_4$. Then $g_\alpha$ (respectively $g_\beta$) is determined by the action of $\bar{g}$ on the set $\{\xi^*+ \alpha_1^*, \xi^*+\alpha_2^*, \xi^*+\alpha_3^*, \xi^*\}$ (respectively $\{\beta_1^*, \beta_2^*, \beta_3^*\}$). 
\item If $\bar{g}(\xi^*) = \xi^* + \beta_j^*$ ($1 \leqslant j \leqslant 3$), then $g_\beta$ maps $\beta_4$ to $\beta_j$ and $g_\alpha$ fixes $\alpha_4$. Then $g_\alpha$ (respectively $g_\beta$) is determined by the action of $\bar{g}$ on the set $\{\alpha_1^*, \alpha_2^*, \alpha_3^*\}$ (respectively $\{\xi^*+ \beta_1^*, \xi^*+\beta_2^*, \xi^*+\beta_3^*, \xi^*\}$). 
\item If $\bar{g}(\xi^*) = \xi^* + \alpha_i^* + \beta_j^*$ ($1 \leqslant i,j \leqslant 3$), then  $g_\alpha$ maps $\alpha_4$ to $\alpha_i$ and $g_\beta$ maps $\beta_4$ to $\beta_j$. Then $g_\alpha$ (respectively $g_\beta$) can be recovered by the action of $\bar{g}$ on the set $\{\xi^*+ \alpha_1^*+ \beta_j^*, \xi^*+\alpha_2^*+ \beta_j^*, \xi^*+\alpha_3^*+ \beta_j^*, \xi^*+ \beta_j^*\}$ (respectively $\{\xi^*+ \beta_1^*+\alpha_i^*, \xi^*+\beta_2^*+\alpha_i^*, \xi^*+\beta_3^*+\alpha_i^*, \xi^*+\alpha_i^*\}$). 
\end{itemize}
\end{proof}

By abuse of notation, we also use $\Sigma_\alpha \times \Sigma_\beta$ to denote its image in $O(q_T) \cong O(q_M)$. There exists a natural exact sequence $1 \rightarrow \widetilde{O}(T) \rightarrow O(T) \stackrel{r_T}{\rightarrow} O(q_T) \rightarrow 1$ which also induces $1 \rightarrow O^*(T) \rightarrow O_-(T) \rightarrow O(q_T) \rightarrow 1$. 
We define $\Gamma \subset O_{-} (T)\subset O(T)$ as the following extension: $$1 \rightarrow O^*(T) \rightarrow \Gamma \rightarrow (\Sigma_\alpha \times \Sigma_\beta) \rightarrow 1.$$ 

By \lemmaref{g_T extension} and \lemmaref{permutation}, the group $\Gamma$ fixes $M$ but may permute the elements $\alpha_1, \ldots, \alpha_4$ (respectively $\beta_1, \ldots, \beta_4$). Note that $O^*(T)$ is a normal subgroup of $\Gamma$ and $\Gamma/O^*(T) = \Sigma_\alpha \times \Sigma_\beta$. Also, $\Gamma$ and $O(T)$ are commensurable and hence $\Gamma$ is an arithmetic group. There is a natural action of $\Gamma$ on $\calD$ (see the description of $\calD$ in \eqref{eq:period-domain2}).

Recall that $\calM_0 \subset \overline{\calM}(1,1)$ is the moduli space of triples $(C, L_1, L_2)$ consisting of a quartic curve $C$ and (labeled) lines $L_1$, $L_2$ such that $C + L_1 + L_2$ has at worst simple singularities. 
\begin{proposition} \label{P birational}
The period map $\calP$ that associates to a (generic) triple $(C, L_1, L_2)$ the periods of the $K3$ surface $S_{(C,L_1,L_2)}$ defines a birational map $\calP: \calM_0 \dashrightarrow \calD/\Gamma$.
\end{proposition}
\begin{proof}
Let $U$ be an open subset of $\calM_0$ parameterizing triples $(C, L_1, L_2)$ with $C$ smooth quartic curves and $L_1$, $L_2$ two (labeled) lines such that $C+L_1+L_2$ has simple normal crossings. Set $\Sigma_4$ to be the permutation group of $4$ elements and note that $\Sigma_\alpha \cong \Sigma_\beta \cong \Sigma_4$. Let $\widetilde{U}$ be the $(\Sigma_4 \times \Sigma_4)$-cover of $U$ that parametrizes quintuples $(C, L_1, L_2, \sigma_1,\sigma_2)$ where $\sigma_k: \{1,2,3,4\} \rightarrow C \cap L_k$ ($k=1,2$) is a labeling of the intersection points of $C \cap L_k$. Note that the monodromy group acts as the permutation group $\Sigma_4$ on the four points of intersection $C \cap L_k$. By \corollaryref{generic primitive}, $\sigma_1$ and $\sigma_2$ determine an $M$-polarization $\jmath$ of the $K3$ surface $S_{(C, L_1, L_2)}$. Therefore there is a well-defined map
$$\widetilde{\calP}: \widetilde{U} \rightarrow \calD/O^*(T).$$
By the global Torelli theorem and the surjectivity of the period map for $K3$ surfaces (see also \propositionref{generic surjectivity}), the map $\widetilde{\calP}$ is a birational morphism. The group $\Sigma_4 \times \Sigma_4$ acts naturally on both $\widetilde{U}$ and $\calD/O^*(T)$ as $\Gamma$ is an extension of $\Sigma_4 \times \Sigma_4$ and $O^*(T)$. Essentially, the actions are induced by the permutation of the labeling of the intersection points $C \cap L_k$ ($k=0,1$). It follows that $\widetilde{\calP}$ is $(\Sigma_4 \times \Sigma_4)$-equivariant and descends to the birational map $\calP: \calM_0 \dashrightarrow \calD/\Gamma$ (see also \lemmaref{g_T extension} and \lemmaref{permutation}).  
\end{proof}

\subsection{M-polarization for non-generic intersections} \label{normalized embedding}

We will show in this section that the birational map $\calP: \calM_0 \dashrightarrow \calD/\Gamma$ in \propositionref{P birational} extends to a birational morphism $\calP: \calM_0 \rightarrow \calD/\Gamma$. To do this, we need to extend the construction of $M$-polarization $\jmath: M \hookrightarrow \Pic(S_{(C, L_1,L_2)})$ to the non-generic triples $(C, L_1, L_2)$ and show that the construction fits in families. The idea is to use the normalized lattice polarization (cf. \cite[Definition 4.24]{laza_n16}) for degree $5$ pairs constructed in \cite[\S 4.2.3]{laza_n16}. A \emph{degree $d$ pair $(D, L)$} consists of a degree $d$ plane curve $D$ and a line $L \subset \bP^2$ (see \cite[Definition 2.1]{laza_n16}). Given a triple $(C, L_1, L_2)$ of a quartic curve $C$ and two different lines $L_1$ and $L_2$, one can construct a degree $5$ pair in two ways: $(C+L_2, L_1)$ or $(C+L_1, L_2)$. We follow the notation of the previous subsections, especially \notationref{lattice M}. We will determine the images of $\gamma, l_1', \alpha_1, \ldots, \alpha_4$ (respectively $\gamma, l_2', \beta_1, \ldots, \beta_4$) using the degree $5$ pair $(C+L_2, L_1)$ (respectively $(C+L_1, L_2)$). 

Let us briefly review the construction of normalized lattice polarization for degree $5$ pairs. See \cite[\S4.2.3]{laza_n16} for more details. Let $(D, L)$ be a degree $5$ pair such that $B:=D+L$ has at worst simple singularities. Let $\bar{S}_{(D,L)}$ be the normal surface obtained as the double cover of $\bP^2$ branched along $B$. Let $S_{(D,L)}$ be the minimal resolution of $\bar{S}_{(D,L)}$, called the $K3$ surface associated to $(D, L)$. The surface $S_{(D,L)}$ can also be obtained as the canonical resolution of $\bar{S}_{(D,L)}$, see \cite[\S III.7]{BPV}. Namely, there exists a commutative diagram: 
$$
\begin{CD}
S_{(D, L)}     @>>>   \bar{S}_{(D,L)}\\
@VV\pi'V        @VVV\\
S'     @>\epsilon>>  \bP^2
\end{CD}
$$
where $S'$ is obtained by an inductive process. Start with $S_{-1} = \bP^2$ and $B_{-1} = B = D+ L$. Simultaneously blow up all the singular points of $B$. Let $\epsilon_0: S_0 \rightarrow \bP^2$ be the resulting surface and set $B_0$ to be the strict transform of $B$ together with the exceptional divisors of $\epsilon_0$ reduced mod $2$. Repeat the process for $S_0$ and $B_0$ until the resulting divisor $B_N$ is smooth. Let $S' = S_N$, $B' = B_N$ and $\epsilon = \epsilon_N \circ \ldots \circ \epsilon_0$. Now take the double cover $\pi': S_{(D, L)} \rightarrow S'$ branched along the smooth locus $B'$. 

The construction of a normalized lattice polarization for degree $5$ pairs is a modification of the process of canonical resolution. We may choose a labeling of the intersection points of $D$ and $L$, which means a surjective map $\sigma: \{0,1,2,3,4\} \rightarrow D \cap L$ satisfying $|\sigma^{-1}(p)| = \mathrm{mult}_p(D \cap L)$ for every $p \in D \cap L$. As argued in \cite[Proposition 4.25]{laza_n16}, $L$ is blown-up exactly $5$ times in the desingularization process described above. The blown-up points can be chosen as the first five steps of the sequence of blow-ups: $S' \rightarrow \ldots S_0 \rightarrow \bP^2$, and the labeling determines the order of these first $5$ blow-ups. Let $\{p_k\}_{k=0}^4$ be the centers of these blow-ups and $E_k$ be the exceptional divisors. Note that $p_k \in S_{k-1}$ (the image of $p_k$ under the contraction to $\bP^2$ is a point of intersection $D \cap L$) and $E_k$ is a divisor on $S_k$ for $k = 0, \ldots, 4$. We define the following divisors $$D_k = (\pi'^*\circ \epsilon_N^* \circ \ldots \epsilon_{k+1}^*)(E_k)$$ for $0 \leqslant k \leqslant 4$. The divisor $D_k$ on $S_{(D,L)}$ is Artin's fundamental cycle (see for example \cite[p. 76]{BPV}) associated to the simple singularity of the curve $B_{k-1}$ at the point $p_k$. 
 
The procedure described above produces $5$ divisors: $D_0, \ldots, D_4$. One can also consider the strict transform of $L$ in $S'$ and take its preimage in $S_{(D,L)}$. This is a smooth rational curve on the $K3$ surface and we will denote its corresponding class by $L'$. We summarize the properties of these $6$ divisors $L', D_0, \ldots, D_4$ in the following result. Given families of curves $(C,L)$, we can carry out a simultaneous resolution in families. As a result the construction above fits well in families, see \cite[p. 2141]{laza_n16}.
\begin{lemma} \label{normalized polarization}
For a pair $(D,L)$  and the surface $S_{(D,L)}$ as described above the following statements hold:
\begin{enumerate}
\item the polarization class of $S_{(D,L)}$ is $(\epsilon \circ \pi')^*\calO_{\bP^2}(1) = 2L' + D_0 + \ldots + D_4$, and
\item their intersections are $(L',L')=-2$, $(D_k,D_{k'}) = -2\delta_{kk'}$, $(L', D_k) = 1$ for $0 \leqslant k,k' \leqslant 4$. 
\end{enumerate}
\end{lemma}
\begin{proof}
See the proof of \cite[Proposition 4.25]{laza_n16}.
\end{proof}

Let us consider triples $(C, L_1, L_2)$ consisting of a quartic curve $C$ and lines $L_1$, $L_2$ such that $C+L_1+L_2$ has worst simple singularities. Let $\bar{S}_{(C,L_1,L_2)}$ be the double plane branched along $C+L_1+L_2$ and $S_{(C,L_1,L_2)}$ be the $K3$ surface obtained by taking the minimal resolution of $\bar{S}_{(C,L_1,L_2)}$. Let $\pi\colon S_{(C,L_1,L_2)} \rightarrow \bP^2$ be the natural morphism. To define a lattice embedding $\jmath: M\hookrightarrow \Pic(S_{(C,L_1,L_2)})$, one needs to specify the images of $\gamma, l_1', \alpha_1, \ldots, \alpha_4, l_2', \beta_1, \ldots, \beta_4$ so that the intersection form is preserved. There is a compatibility condition induced by $2l_1' + \alpha_1 + \ldots + \alpha_4 = 2l_2' + \beta_1 + \ldots + \beta_4$. Recall that $h=\xi + \gamma = 2l_1' + \alpha_1 + \ldots + \alpha_4 +\gamma = 2l_2' + \beta_1 + \ldots + \beta_4 + \gamma$. We also require that $\jmath(h)$ is the class of the base point free polarization $\pi^*\calO_{\bP^2}(1)$.   

Given a triple $(C, L_1, L_2)$, one has two associated degree $5$ pairs: $(C+L_2, L_1)$ and $(C+L_1, L_2)$, which induce the same K3 surface $S_{(D,L)}$, constructed as above. Let us fix two labellings $\sigma_1: \{0,1,2,3,4\} \rightarrow C \cap L_1$ and $\sigma_2: \{0,1,2,3,4\} \rightarrow C \cap L_2$ such that $\sigma_1(0) = \sigma_2(0) = L_1 \cap L_2$. Every degree $5$ pair produces $6$ divisors as described above. For the pair $(C+L_2, L_1)$ (respectively $(C+L_1, L_2)$), we denote the $6$ divisors by $L_1', R_0, \ldots, R_4$ (respectively $L_2', T_0, \ldots, T_4$). Note that $\pi^*\calO_{\bP^2}(1) = 2L_1' + R_0 + \ldots + R_4 = 2L_2' + T_0 + \ldots + T_4$. We define $\jmath: M\hookrightarrow \Pic(S_{(C, L_1, L_2)})$ as follows:
\begin{itemize}
\item $\jmath(\gamma) = R_0 = T_0$ (by our choice of the labellings, both $R_0$ and $T_0$ are the fundamental cycle associated to the singularity of $C+L_1+L_2$ at the point $L_1 \cap L_2$);
\item $\jmath(l_1') = L_1'$ and $\jmath(l_2') = L_2'$;
\item $\jmath(\alpha_i) = R_i$ and $\jmath(\beta_j) = T_j$ for $1 \leqslant i,j \leqslant 4$.
\end{itemize}

After the first blow-up $\epsilon_0$, the strict transforms of the two lines $L_1$ and $L_2$ are disjoint and hence, we have $(R_i, T_j) = 0$ for $1 \leqslant i,j \leqslant 4$. Using \lemmaref{normalized polarization}, it is straightforward to verify that $\jmath$ is a well-defined lattice embedding. The embedding $\jmath$ also satisfies the following geometric properties and fits well in families (cf. \cite[\S4.2.3]{laza_n16}, especially the last paragraph on page $2141$). 
\begin{enumerate}
\item $\jmath(h)$ is the class of the base point free polarization $\pi^*\calO_{\bP^2}(1)$.
\item $\jmath(l_1')$ and $\jmath(l_2')$ are the classes of irreducible rational curves.
\item $\jmath(\gamma), \jmath(\alpha_1), \ldots, \jmath(\alpha_4)$ (respectively $\jmath(\gamma), \jmath(\beta_1), \ldots, \jmath(\beta_4)$) are classes of effective divisors which are contracted by $\pi$ to the points of the intersection $C \cap L_1$ (respectively $C \cap L_2$). In particular, $\jmath(\gamma)$ is contracted to the point $L_1 \cap L_2$.
\end{enumerate}

To conclude that $\jmath$ is an $M$-polarization we also need the following lemma.
\begin{lemma}
The lattice embedding $\jmath: M \hookrightarrow \Pic(S_{(C, L_1, L_2)})$ is primitive. 
\end{lemma}
\begin{proof}
This follows from a case by case analysis. Specifically, we check the conditions of \lemmaref{primitive} as in the proof of \propositionref{primitive_irreducible}. 
\end{proof}

\begin{proposition} \label{P injective}
The birational map in \propositionref{P birational} extends to a morphism $\calP: \calM_0 \rightarrow \calD/\Gamma$. Moreover, the map $\calP$ is injective. 
\end{proposition}
\begin{proof}
Given a triple $(C, L_1, L_2)$ consisting of a quartic curve $C$ and lines $L_1$, $L_2$ such that $C + L_1 + L_2$ has at worst simple singularities, we consider the $M$-polarized $K3$ surface $(S_{(C, L_1, L_2)}, \jmath)$ where $S_{(C,L_1,L_2)}$ is the $K3$ surface obtained by taking the minimal resolution of the double plane branched along $C + L_1 +L_2$ and $\jmath$ is the lattice polarization constructed above. By \cite{dolgachev_MSK3}, the $M$-polarized $K3$ surface $(S_{(C, L_1, L_2)}, \jmath)$ corresponds to a point in $\calD/O^*(T)$. The polarization $\jmath$ depends only on $(C, L_1, L_2)$, the ordering of $C \cap L_1$ and the ordering of $C \cap L_2$, and it is compatible with the action of $\Gamma/O^*(T) = \Sigma_4 \times \Sigma_4$. Consequently, we can associate to every triple $(C, L_1, L_2)$ a point in $\calD/\Gamma$. In other words, we have a well-defined morphism $\calP: \calM_0 \rightarrow \calD/\Gamma$ extending the birational map in \propositionref{P birational}. 

Choose a point $\omega \in \calD/\Gamma$ (more precisely, $\omega$ is a $\Gamma$-orbit) which corresponds to an $M$-polarization $K3$ surface $S_{(C, L_1, L_2)}$. \lemmaref{g_T extension} allows us to extend an element of $\Gamma$ to an isometry of the $K3$ lattice $\Lambda_{K3}$. The global Torelli theorem for $K3$ surfaces implies that the period $\omega$ determines the isomorphism class of the $K3$ surface $S_{(C + L_1 + L_2)}$. By our construction the classes $h$, $l_1'$ and $l_2'$ are fixed by $\Gamma$. It follows that the period point $\omega$ uniquely expresses the $K3$ surface as a double cover of $\bP^2$ and determines two line components of the branched locus. Now we conclude that $\omega$ determines uniquely the triple $(C, L_1, L_2)$. 
\end{proof}

\subsection{Surjectivity of the period map} \label{sec:surjectivity}
We will show in this section that the period map $\calP: \calM_0 \rightarrow \calD/\Gamma$ is surjective. Given the general result of surjectivity of period maps for (lattice) polarized $K3$, one has to establish that any $K3$ surface carrying an $M$-polarization is of type $S_{(C, L_1, L_2)}$. 

\begin{proposition} \label{P surjective}
Let $S$ be a $K3$ surface such that there exists a primitive embedding $\jmath: M \hookrightarrow \Pic(S)$. Then there exists a plane quartic curve $C$ and two different lines $L_1, L_2 \subset \bP^2$ such that $S \cong S_{(C, L_1, L_2)}$ and $C+L_1+L_2$ has at worst simple singularities.
\end{proposition}
\begin{proof}
We apply \cite[Proposition 4.31]{laza_n16} (see also \cite[Lemmas 4.27, 4.28, 4.30]{laza_n16}). The idea is to consider the primitive sublattices $M_1$ and $M_2$ of $M$ generated by $l_1', \gamma, \alpha_1, \ldots, \alpha_4$ and $l_2', \gamma, \beta_1, \ldots, \beta_4$ respectively. Both of the sublattices have the following Gram matrix
$$
 \left( {\begin{array}{cccccc}
   -2 & 1 & 1 & 1 & 1 & 1   \\
   1 & -2 & 0 & 0 & 0 & 0   \\
   1 & 0 & -2 & 0 & 0 & 0   \\
   1 & 0 & 0 & -2 & 0 & 0   \\
   1 & 0 & 0 & 0 & -2 & 0   \\
   1 & 0 & 0 & 0 & 0 & -2   \\
   \end{array} } \right)_.
$$
Hence they are isomorphic to the lattice considered in \cite[Notation 4.11]{laza_n16} for degree $5$ pairs.  In particular, $S$ is both $M_1$-polarized and $M_2$-polarized. Indeed, recall that
\begin{equation}
h = 2l_1' + \alpha_1 + \ldots + \alpha_4 + \gamma = 2l_2' + \beta_1 + \ldots + \beta_4 + \gamma,
\label{eq:expression_h}
\end{equation}
which coincides with Laza's lattice (for both $M_1$ and $M_2$). Using \cite[Proposition 4.31]{laza_n16}, we find two degree $5$ pairs $(D_1, L_1)$ and $(D_2, L_2)$ (where, \emph{a priori}, $D_1$ and $D_2$ may be irreducible) such that $D_1 + L_1 = D_2 + L_2$ has at worst simple singularities. The two morphisms $S \rightarrow \bP^2$ associated to each degree $5$ pair are both defined by $\jmath(h)$, and hence they are the same. Because $L_1$ and $L_2$ are both contained in the branch locus of the map $S \rightarrow \bP^2$, $D_1=L_2+C$ and $D_2=L_1+C$, where $C$ is a quartic plane curve such that $C+L_1+L_2$ has at worst simple singularities. 
\end{proof}

\begin{theorem} \label{P isomorphic}
Consider the triples $(C, L_1, L_2)$ consisting of a quartic curve $C$ and lines $L_1$, $L_2$ such that $C + L_1 + L_2$ has at worst simple singularities. Let $S_{(C,L_1,L_2)}$ be the $K3$ surface obtained by taking the minimal resolution of the double plane branched along $C + L_1 +L_2$. The birational map sending $(C, L_1, L_2)$ to the periods of $S_{(C,L_1,L_2)}$ in \propositionref{P birational} extends to an isomorphism $\calP: \calM_0 \rightarrow \calD/\Gamma$.
\end{theorem}
\begin{proof}
It suffices to prove that $\calP$ is surjective. Let $\omega \in \calD/\Gamma$ be a period point. By the surjectivity of the period map of lattice-polarized $K3$ surfaces (see \cite[Theorem 3.1]{dolgachev_MSK3}) there exists an $M$-polarized $K3$ surface $(S, \jmath: M \hookrightarrow \Pic(S))$ corresponding to $\omega$. By \propositionref{P surjective} the $K3$ surface $S$ is the double cover of $\mathbb P^2$ branched at a plane quartic curve $C$ and two different lines $L_1$, $L_2$. Moreover, let $M_1$ and $M_2$ be the primitive sublattices of $M$ defined in the proof of \propositionref{P surjective}. After choosing the K\"ahler cone $V^+(S)$ and the set of effective $(-2)$ curves $\Delta^+(S)$ as in the proof of \cite[Theorem 4.1]{laza_n16}, we may assume that the restrictions $\jmath|_{M_1}$ and $\jmath|_{M_2}$ of the polarization $\jmath$ to the sublattices $M_1$ and $M_2$ are both normalized embeddings, as defined in \cite[Definition 4.24]{laza_n16}. The embeddings $\jmath|_{M_1}$ and $\jmath|_{M_2}$ are unique up to permutation of the classes $\alpha_1,\ldots, \alpha_4$ ($\beta_1,\ldots, \beta_4$, respectively) thanks to \cite[Lemma 4.29]{laza_n16}. It follows that the polarization $\jmath$ is unique up to action of $\Sigma_4\times\Sigma_4$ and coincides with our construction in \subsectionref{normalized embedding}. By \propositionref{P birational} and \propositionref{P injective} the period map $\calP: \calM_0 \rightarrow \calD/\Gamma$ is a bijective birational morphism between normal varieties. As a result, $\calP$ is an isomorphism, by Zariski's Main Theorem.
\end{proof}

\subsection{The period map for unlabeled triples} \label{unlabeled P'}
Consider the compact space $\mathcal M^*$
and $\calM_0^* \subset \overline{\calM}^*$, consisting on the subset of triples $(C, L, L')$ formed by a quartic curve $C$ and unlabeled lines $L$, $L'$ such that the sextic curve $C + L + L'$ is reduced and has at worst simple singularities, as constructed in \corollaryref{cor:gitcll'}. In this subsection we define the period map $\calP': \calM_0^* \rightarrow \calD/\Gamma'$ ---for an appropriately chosen arithmetic group $\Gamma'$--- and show that $\calP'$ is an isomorphism. We use the same approach taken to define $\calP: \calM_0 \rightarrow \calD/\Gamma$ (\propositionref{P birational} and \propositionref{P injective}) and to prove \theoremref{P isomorphic}. The modification one needs to do is to choose a different arithmetic group $\Gamma'$. We follow the same notation as in the previous subsections, especially regarding the description of the subgroup $\Sigma_\alpha \times \Sigma_\beta$ in \lemmaref{permutation}. Consider the subgroup $\Sigma_\alpha \times \Sigma_\beta \rtimes \bZ/2\bZ \subset O(M)$ where the factor $\bZ/2\bZ$ corresponds to the swap of $\alpha_i$'s and $\beta_i$'s for $0 \leqslant i \leqslant 4$ (the induced action on $A_M$ exchanges $\alpha_i^*$ with $\beta_i^*$ for $1 \leqslant i \leqslant 3$ and fixes $\gamma^*$ and $\xi^*$). As in the proof of \lemmaref{permutation}, we can verify that the composition $\Sigma_\alpha \times \Sigma_\beta \rtimes \bZ/2\bZ  \hookrightarrow O(M) \rightarrow O(q_M)$ is injective. Now we define $\Gamma'$ to be the following extension: $$1 \rightarrow O^*(T) \rightarrow \Gamma' \rightarrow (\Sigma_\alpha \times \Sigma_\beta) \rtimes \bZ/2\bZ \rightarrow 1.$$ 

For $(C,L,L') \in \calM_0^*$ (such that $C+L+L'$ is reduced and has at worst simple singularities) we consider the period of the $K3$ surface $S_{(C,L,L')}$ which is the minimal resolution of the double cover of $\bP^2$ branched along $C+L+L'$. Because $S_{(C,L,L')}$ is polarized by the lattice $M$, the period corresponds to a point in $\calD$. The lattice polarization depends on the labeling of $L$ and $L'$, the ordering of $C \cap L$ and the ordering of $C \cap L'$, and thus is compatible with the action of $\Gamma'/O^*(T) = \Sigma_4 \times \Sigma_4 \rtimes \bZ/2\bZ$. Therefore we have a well-defined period map $\calP': \calM_0^* \rightarrow \calD/\Gamma'$ (see also \propositionref{P birational} and \propositionref{P injective}). Moreover, the same argument in \propositionref{P injective} and \theoremref{P isomorphic} allows us to prove that the period map $\calP': \calM_0^* \rightarrow \calD/\Gamma'$ is an isomorphism. 

\subsection{Comparison of the GIT and the Baily-Borel compactifications} \label{sec:GITBB}

Consider the moduli space $\calM_0^* \subset \overline{\calM}^*$ of triples $(C, L, L')$ formed by a quartic curve $C$ and unlabeled lines $L$, $L'$ such that the sextic curve $C + L + L'$ has at worst simple singularities. We have constructed a period map $\calP': \calM_0^* \rightarrow \calD/\Gamma'$ in \subsectionref{unlabeled P'} and have shown that it is an isomorphism. There are two natural ways to compactify $\calM_0^*$ as the GIT quotient $\overline{\calM}^*\coloneqq\overline{\calM}^*_{1,4,2}$ defined in \eqref{eq:GIT_unlabeled} and described in \corollaryref{cor:gitcll'}, or as the Baily-Borel compactification \cite{bailyborel}. We compare these two compactifications by applying some general results of Looijenga \cite{looijenga}. See also \cite[Theorem 4.2]{laza_n16}. 
\begin{theorem} \label{P'gitbb}
The period map $\calP': \calM_0^* \rightarrow \calD/\Gamma'$ extends to an isomorphism of projective varieties $\overline{\calP'}: \overline{\calM}^* \stackrel{\cong}{\rightarrow} (\calD/\Gamma')^*$ where $(\calD/\Gamma')^*$ denotes the Baily-Borel compactification of $\calD/\Gamma'$.  
\end{theorem}
\begin{proof}
We apply a general framework of comparing GIT compactifications to certain compactifications of the period domain developed by Looijenga. Specifically, by \cite[Theorem 7.6]{looijenga} an isomorphism $\calM \cong (\Omega \backslash \calH)/\Lambda$ (typically coming from a period map) between a geometric quotient $\calM$ and a complement of an arithmetic hyperplane arrangement $\calH$ in a type IV domain $\Omega$ extends to an isomorphism $\overline{\calM} \cong \widetilde{\Omega/\Lambda}$ between the GIT compactification $\overline{\calM}$ and the Looijenga compactification $\widetilde{\Omega/\Lambda}$ associated to $\calH$ if their polarizations agree and $\dim(\overline{\mathcal M})-\dim(\mathcal M)>1$. We have $\calM_0^* \cong \calD/\Gamma'$. The hyperplane arrangement is empty and the associated Looijenga compactification is the Baily-Borel compactification $(\calD/\Gamma')^*$. Moreover, by \corollaryref{cor:gitcll'} and \lemmaref{corollary:dimension-moduli}, we have
$$\dim(\overline{\mathcal M}^*)-\dim(\mathcal M^*_0)=\dim(\overline{\mathcal M}^*)-1>1,$$
and their polarizations agree by restriction of the isomorphic polarizations for the GIT of sextic curves and for the compact moduli of K3 surfaces of degree $2$ (see \cite[\S8]{looijenga}). Hence, by \cite[Theorem 7.6]{looijenga} $\calP': \calM_0^* \rightarrow \calD/\Gamma'$ is an isomorphism for polarized varieties. 
\end{proof}


\begin{question}
\label{que:isomorphism-before-quotient}
Does the period map for labeled triples $\calP: \calM_0 \rightarrow \calD/\Gamma$ (cf. \theoremref{P isomorphic}) preserve the natural polarizations?
\end{question}
A positive answer to this question would imply that the period map $\calP: \calM_0 \rightarrow \calD/\Gamma$ can be extended to an isomorphism $\overline{\calP}: \overline{\calM}(1,1) \rightarrow (\calD/\Gamma)^*$. We strongly believe that the answer is yes (by pulling back the polarizations for sextic curves and degree $2$ $K3$ surfaces via the double covers $\overline{\calM}(1,1) \rightarrow \overline{\calM}^*$ and $\calD/\Gamma \rightarrow \calD/\Gamma'$). 

\subsection{The Baily-Borel compactification} \label{sec:BB}

The locally symmetric space $\calD/\Gamma'$ admits a canonical minimal compactification, the Baily-Borel compactification $(\calD/\Gamma')^*$ (cf. \cite{bailyborel}). The boundary components of $(\calD/\Gamma')^*$ are either $0$-dimensional (Type III components) or $1$-dimensional (Type II components), and they correspond to the primitive rank $1$, respectively, rank $2$ isotropic sublattices of $T$ up to $\Gamma'$-equivalence. Following the approach of \cite{scattone}, \cite{sterk_enriques} and \cite{laza_n16}, we determine the number of the Type III boundary components of $(\calD/\Gamma')^*$ and compute certain invariants for the Type II boundary components. Notice that by Theorem \theoremref{P'gitbb}, the number of these boundary components and some of their invariants (such as the dimension) can be worked out from the boundary components of the GIT quotient $\overline{\calM}^*$ described in Corollary \ref{cor:gitcll'}.

We determine the $0$-dimensional components of $(\calD/\Gamma')^*$ using \cite[Proposition 4.1.3]{scattone}. The $0$-dimensional boundary components are in one-to-one correspondence with the $\Gamma'$-orbits of primitive isotropic rank $1$ sublattices of $T$. By \propositionref{determine M T} and \cite[Theorem 3.6.2]{nikulin} we have $T \cong U \perp U(2) \perp A_1^2 \perp D_6 \cong U \perp U \perp A_1^4 \perp D_4$. (In particular, $T$ contains two hyperbolic planes.) Write $\Gamma^* = \Sigma_\alpha \times \Sigma_\beta \rtimes \bZ/2\bZ \subset O(q_T)$ (see \subsectionref{unlabeled P'}). Note that for $v \in T$ one can associate a vector $\bar{v} \in A_T=T^*/T$ defined by $\bar{v} \equiv \frac{v}{\mathrm{div}(v)} \mod T$ (where $\mathrm{div}(v)$ is the divisor of $v$ which is a positive integer such that $(v,T)=\mathrm{div}(v)\bZ$). If $v$ is a primitive isotropic vector then $\bar{v}$ is an isotropic element in $A_T$. By \cite[Proposition 4.1.3]{scattone} the map $\bZ v \mapsto \bar{v}$ induces a bijection between the equivalence classes of primitive isotropic rank $1$ sublattices of $T$ and $\Gamma^*$-orbits of isotropic elements of $A_T$. Because $T$ is the orthogonal complement of $M$ in $\Lambda_{K3}$, one has $(A_M, q_M) \cong (A_T, -q_T)$. 
We have computed the discriminant quadratic form $q_M$ in \lemmaref{A_M}. In particular, there are $20$ isotropic elements in $A_M \cong A_T$: $0, \gamma^*$, $\alpha_i^*+\beta_j^*$ and $\alpha_i^*+\beta_j^*+\gamma^*$ ($1 \leqslant i,j \leqslant 3$). The action of $\Gamma^*$ has been described in the proof of \lemmaref{permutation} and \subsectionref{unlabeled P'}. It is easy to see that $\alpha_i^*+\beta_j^*$ and $\alpha_i^*+\beta_j^*+\gamma^*$ ($1 \leqslant i,j \leqslant 3$) form one $\Gamma^*$-orbit. As a result, the Baily-Borel compactification $(\calD/\Gamma')^*$ consists of three $0$-dimensional boundary components.   

\begin{remark}
Similarly, one can show that $(\calD/\Gamma)^*$ has three $0$-dimensional boundary components (compare \lemmaref{lemma:Bd11}).\end{remark}

\begin{remark}
As discussed in \cite[\S 4.4.1]{laza_n16}, one important invariant for the $O_-(T)$-equivalence class of isotropic sublattices $E$ of $T$ is the isomorphism classes of $E^{\perp}/E$ (and we shall use it to label $E$). Let us compute the isomorphism classes of $v^{\perp}/\bZ v$ (where $\bZ v$ is a primitive isotropic rank $1$ sublattice of $T$). Observe that $T \cong U \perp M$. One could compute the Gram matrix of $v^{\perp}/\bZ v$ explicitly. Alternatively, we consider $H_v:=\bZ v_{T^*}^{\perp \perp}/\bZ v$ (cf. \cite[\S4.4.1]{laza_n16}) which is an isotropic subgroup of $A_T \cong (\bZ/2\bZ)^6$ and the discriminant group $A_{v^{\perp}/\bZ v} \cong H_v^{\perp}/H_v$. In our case, $H_v$ equals either $0$ or $\bZ/2\bZ$. The lattice $v^{\perp}/\bZ v$ is an even hyperbolic (N.B. the signature is $(1,9)$) $2$-elementary lattice. By a direct computation we get the following (see also \cite[Theorem 3.6.2]{nikulin}).
\begin{itemize}
\item If $\bar{v}=0$, then $v^{\perp}/\bZ v \cong U \perp A_1^4 \perp D_4 \cong U(2) \perp A_1^2 \perp D_6$.
\item If $\bar{v}=\gamma^*$, then $v^{\perp}/\bZ v  \cong U \perp D_4 \perp D_4 \cong U(2) \perp D_8$.
\item If $\bar{v}=\alpha_i^*+\beta_j^*$ or $\alpha_i^*+\beta_j^*+\gamma^*$ ($1 \leqslant i,j \leqslant 3$), then $v^{\perp}/\bZ v \cong U \perp A_1^2 \perp D_6 \cong U(2) \perp A_1 \perp E_7$.
\end{itemize}
\end{remark}

To determine the $1$-dimensional components of $(\calD/\Gamma')^*$, one needs to compute the equivalence classes of primitive isotropic rank $2$ sublattices of $T$. We use the algorithm for classifying isotropic vectors in hyperbolic lattices due to Vinberg \cite{vinberg}. Specifically, for each of the equivalence classes of primitive isotropic rank $1$ sublattices $\bZ v$ of $T$ we apply Vinberg's algorithm to the hyperbolic lattice $v^{\perp}/\bZ v$ (with respect to the action by the stabilizer $\Gamma'_v$ of $v$). 

Now we briefly recall Vinberg's algorithm \cite{vinberg} (see also \cite[\S4.3]{sterk_enriques}). Let $N$ be a hyperbolic lattice of signature $(1,n)$. (In our case we take $N=v^{\perp}/\bZ v$.) The algorithm starts by fixing an element $h \in N$ of positive square. Then one needs to inductively choose roots $\delta_1, \delta_2, \ldots$ such that the distance function $\frac{(h,\delta)^2}{|(\delta,\delta)|}$ is minimized. The algorithm stops with the choice of $\delta_N$ if every connected parabolic subdiagram (i.e. the extended Dynkin diagram of a root system) of the Dynkin diagram $\Sigma$ associated to the roots $\delta_1, \delta_2, \ldots, \delta_N$ is a connected component of some parabolic subdiagram of rank $n-1$. If the algorithm stops then the $W(N)$-orbits of the isotropic lines in $N$ correspond to the parabolic subdiagrams of rank $n-1$ of $\Sigma$ (N.B. the isomorphism classes of $E^{\perp}/E$, where $E$ is an isotropic rank $2$ sublattice of $T$ containing $v$, are determined by the Dynkin diagrams of the parabolic subdiagrams). To determine the equivalence classes of the isotropic vectors by a larger group which contains the Weyl group $W(N)$ as a subgroup of finite index, one should take certain symmetries of $\Sigma$ into consideration. 

In our case, a straightforward application of Vinberg's algorithm allows us to compute the isomorphism classes of $E^{\perp}/E$.  
\begin{itemize}
\item If $\bar{v}=0$, then $v^{\perp}/\bZ v$ has at least three equivalence classes of isotropic vectors which correspond to $A_1^4 \perp D_4$, $A_1^2 \perp D_6$ and $D_4 \perp D_4$ respectively.  
\item If $\bar{v}=\gamma^*$, then $v^{\perp}/\bZ v$ has at least two equivalence classes of isotropic vectors which correspond to $D_4 \perp D_4$ and $D_8$ respectively.
\item If $\bar{v}=\alpha_i^*+\beta_j^*$ or $\alpha_i^*+\beta_j^*+\gamma^*$ ($1 \leqslant i,j \leqslant 3$), then $v^{\perp}/\bZ v$ has at least three equivalence classes of isotropic vectors which correspond to $A_1^2 \perp D_6$, $A_1 \perp E_7$  and $D_8$ respectively.
\end{itemize}
By \theoremref{P'gitbb} and \corollaryref{cor:gitcll'} we conclude that the Baily-Borel compactification $(\calD/\Gamma')^*$ consists of five $1$-dimensional boundary components labeled by $A_1^4 \perp D_4$, $A_1^2 \perp D_6$, $A_1 \perp E_7$, $D_4 \perp D_4$ and $D_8$ respectively.

\begin{remark}
Using the Clemens-Schmid exact sequence and the incidence relation of the GIT boundary components (see also \cite[Theorem 4.32]{laza_n16}), we match the GIT boundary of $\overline{\calM}^*$ in \corollaryref{cor:gitcll'} with the Baily-Borel boundary of $(\calD/\Gamma')^*$.

\begin{center}
\begin{tabular}{cc}
\hline
GIT boundary & Baily-Borel boundary  \\
\hline 
$\overline{II}(1)$ & $D_4 \perp D_4$ \\
$\overline{II}(2a1)$ & $A_1 \perp E_7$ \\
$\overline{II}(2a2)$ & $A_1^4 \perp D_4$ \\
$\overline{II}(2b)$ & $A_1^2 \perp D_6$\\
$\overline{II}(3)$ & $D_8$\\
\hline
$\overline{III}(1)$ & $U \perp D_4 \perp D_4$ \\
$\overline{III}(2a)$ & $U \perp A_1^2 \perp D_6$ \\
$\overline{III}(2b)$ & $U \perp A_1^4 \perp D_4$\\
\hline
\end{tabular}
\end{center}
\end{remark}

\bibliography{ref}

\end{document}